\pdfoutput=1
\RequirePackage{ifpdf}
\ifpdf 
\documentclass[pdftex]{sigma}
\else
\documentclass{sigma}
\fi

\usepackage[all]{xy}
\usepackage{quiver}
\usepackage{tikz}
\usepackage{adjustbox}
\usepackage{hhline}

\numberwithin{equation}{section}

\newtheorem{Theorem}{Theorem}[section]
\newtheorem*{Theorem*}{Theorem}
\newtheorem{Corollary}[Theorem]{Corollary}
\newtheorem{Lemma}[Theorem]{Lemma}
\newtheorem{Proposition}[Theorem]{Proposition}
 { \theoremstyle{definition}

\newtheorem{Example}[Theorem]{Example}
\newtheorem{Remark}[Theorem]{Remark} }

\pgfmathdeclarefunction{n}{0}{}
\def\rdots{\reflectbox{$\ddots$}}

\newcommand{\breakingcomma}{%
 \begingroup\lccode`~=`,
 \lowercase{\endgroup\expandafter\def\expandafter~\expandafter{~\penalty0 }}}

\newcommand{\CC}{{\mathbb{C}}}

\newcommand{\PP}{{\mathbb{P}}}
\newcommand{\QQ}{{\mathbb{Q}}}

\newcommand{\RR}{{\mathbb{R}}}
\newcommand{\ZZ}{{\mathbb{Z}}}

\newcommand{\NN}{{\mathbb{N}}}

\newcommand{\calS}{{\mathcal S}}
\newcommand{\eps}{\varepsilon}

\newcommand{\SL}{{\mathrm{SL}}}
\newcommand{\GL}{{\mathrm{GL}}}
\newcommand{\epi}{{\bf e}}

\newcommand{\Jac}{\mathrm{Jac}}
\newcommand{\Fix}{\mathrm{Fix}}

\newcommand{\id}{\mathrm{id}}

\newcommand{\bx}{{\bf x}}

\newcommand{\hess}{\mathrm{hess}}
\newcommand{\age}{\mathrm{age}}

\newcommand{\QR}[2]{
\left.\raisebox{0.5ex}{\ensuremath{#1}}
\ensuremath{\mkern-3mu}\Big/\ensuremath{\mkern-3mu}
\raisebox{-0.5ex}{\ensuremath{#2}}\right.}

\def\A{{\mathcal A}}
\def\B{{\mathcal B}}

\def\p{\partial }

\newcommand\iot{\iota}

\newcommand\wt{\widetilde}

\begin{document}
\allowdisplaybreaks

\newcommand{\arXivNumber}{2307.01295}

\renewcommand{\PaperNumber}{024}

\FirstPageHeading

\ShortArticleName{Hodge Diamonds of the Landau--Ginzburg Orbifolds}
\ArticleName{Hodge Diamonds of the Landau--Ginzburg Orbifolds}

\Author{Alexey BASALAEV~$^{\rm ab}$ and Andrei IONOV~$^{\rm c}$}
\AuthorNameForHeading{A.~Basalaev and A.~Ionov}

\Address{$^{\rm a)}$~Faculty of Mathematics, National Research University Higher School of Economics,\\
\hphantom{$^{\rm a)}$}~6 Usacheva Str., 119048 Moscow, Russia}
\Address{$^{\rm b)}$~Skolkovo Institute of Science and Technology, 3~Nobelya Str., 121205 Moscow, Russia}
\EmailD{\href{mailto:a.basalaev@skoltech.ru}{a.basalaev@skoltech.ru}}

\Address{$^{\rm c)}$~Boston College, Department of Mathematics, Maloney Hall, Fifth Floor, \\
\hphantom{$^{\rm c)}$}~Chestnut Hill, MA 02467-3806, USA}
\EmailD{\href{mailto:ionov@bc.edu}{ionov@bc.edu}}

\ArticleDates{Received July 12, 2023, in final form March 06, 2024; Published online March 25, 2024}

\Abstract{Consider the pairs $(f,G)$ with $f = f(x_1,\dots,x_N)$ being a polynomial defining a~quasihomogeneous singularity and $G$ being a~subgroup of~${\rm SL}(N,\mathbb{C})$, preserving $f$. In particular, $G$ is not necessary abelian. Assume further that $G$ contains the grading operator~$j_f$ and $f$ satisfies the Calabi--Yau condition. We prove that the nonvanishing bigraded pieces of the B-model state space of $(f,G)$ form a diamond. We identify its topmost, bottommost, leftmost and rightmost entries as one-dimensional and show that this diamond enjoys the essential horizontal and vertical isomorphisms.}

\Keywords{singularity theory; Landau--Ginzburg orbifolds}

\Classification{32S05; 14J33}

\section{Introduction}

Let a polynomial $f \in \CC[x_1,\dots,x_N]$ be quasihomogeneous with respect to some positive integers~${d_0,d_1,\dots,d_N}$, i.e.,
\[
	f\bigl(\lambda^{d_1}x_1,\dots,\lambda^{d_N}x_N\bigr) = \lambda^{d_0} f(x_1,\dots,x_N), \qquad \forall \lambda \in \CC^\ast.
\]
Assume also that $x_1=\cdots=x_N = 0$ is the only critical point of $f$ and $d_1,\dots,d_N$ have no common factor. Then the zero set ${f(x_1,\dots,x_N)=0}$ defines a degree $d_0$ quasismooth hypersurface $X_f$ in $\PP(d_1,\dots,d_N)$. Such hypersurfaces became of great interest in the early 90's in the context of mirror symmetry (cf.~\cite{BH95,BH93}). In particular, if the \textit{Calabi--Yau condition} $d_0 = \sum_{k=1}^N d_k$ holds, then the first Chern class of $X_f$ vanishes and, hence, its canonical bundle is trivial meaning that~$X_f$ is a Calabi--Yau variety.

The polynomials $f$ above define the so-called \textit{quasihomogeneous singularities} and can be studied from the point of view of singularity theory. The varieties $X_f$ at the same time are the objects of K\"ahler geometry. To relate the singularity theory properties of $f$ to the K\"ahler geometry properties of $X_f$ is an important problem. This problem is in particular interesting in the context of mirror symmetry.

\subsection{Hodge diamonds of Calabi--Yau manifolds}

The state space of a Calabi--Yau manifolds $X$ is the cohomology ring $H^\ast(X)$.
This cohomology ring is naturally bigraded building up a Hodge diamond of size $D := \dim_\CC X$. In particular the following properties hold:
\begin{enumerate}\itemsep=0pt
 \item[(1)] $H^\ast(X) = \bigoplus_{p,q \in \ZZ} H^{p,q}(X)$,
 \item[(2)] $\dim H^{p,q}(X) = 0$ if $p <0$ or $q <0$ or $p > D$ or $q > D$,
 \item[(3)] $\dim H^{0,0}(X) = \dim H^{D,D}(X) = 1$,
 \item[(4)] $\dim H^{D,0}(X) = \dim H^{0,D}(X) = 1$,
 \item[(5)] there is a ``horizontal'' vector space isomorphism $H^{p,q}(X) \cong H^{q,p}(X)$,
 \item[(6)] there is a ``vertical'' vector space isomorphism $H^{p,q}(X) \cong \bigl(H^{D-q,D-p}(X)\bigr)^\vee$, where $(-)^\vee$ stands for the dual vector space
\end{enumerate}

In physics, this bigrading is coming from considering Calabi--Yau B-model associated to $X$ and A-model bigrading is obtained from it by the so-called {\itshape rotation of the diamond by $90^\circ$}.
On the level of state spaces the mirror symmetry map is an isomorphism of the cohomology vector spaces for a {\itshape dual} pair of Calabi--Yau manifolds switching A- and B-model bigradings.

More, generally if considering Calabi--Yau orbifolds in place of quasismooth varieties, one replaces the ordinary cohomology ring by the Chen--Ruan cohomology ring $H^\ast_{\rm orb}$.
This is an essential question if $H^\ast_{\rm orb}$ forms a Hodge diamond too. Some of the Hodge diamond properties above follow directly from the definitions or could be verified in the similar way to our main Theorem~\ref{theorem: main} below, while the others (like the property (4)) were not investigated in literature up to our knowledge and do not look to be straightforward.

\subsection{Landau--Ginzburg orbifolds}
Another facet of mirror symmetry is given by matching the so-called Landau--Ginzburg orbifolds in place of Calabi--Yau manifolds or orbifolds (cf. \cite{IV90,Kreu94,S20,V89,W93}). Mathematically, these are the pairs $(f,G)$ with $f$ being a quasihomogeneous polynomial with the only critical point $0 \in \CC^N$ and $G$ being a group of symmetries of~$f$.

Consider the \textit{maximal group of linear symmetries} of $f$
\[
	\GL_f := \left\lbrace g \in \GL(N,\CC) \mid f(g \cdot \bx ) = f(\bx) \right\rbrace.
\]
It is nontrivial because it contains a nontrivial subgroup $J$ generated by $j_f$
\[
	j_f \cdot (x_1,\dots,x_N) := \bigl({\rm e}^{2 \pi \sqrt{-1} d_1/d_0} x_1,\dots, {\rm e}^{2 \pi \sqrt{-1} d_N/d_0}x_N\bigr).
\]
Also important is the group $\SL_f := \GL_f \cap \SL(N,\CC)$ consisting of $\GL_f$ elements preserving the volume form of $\CC^N$.

For any $G \subseteq \GL_f$, the pair $(f,G)$ is called a \textit{Landau--Ginzburg orbifold}. One associates to it the \textit{state space}, which is the $G$-equivariant generalization of a Jacobian ring of~$f$, together with A- and B- model bigradings, which again differ by $90^\circ$-rotation from one another provided~$G$ acts by transformations with determinant~1. Since these are interdependent, within this paper we will focus on just the B-model, as it has clearer geometric interpretation, much alike bigrading on the cohomology of Calabi--Yau manifold. For this reason, we will call this state space endowed with B-model bigrading by~$\B(f,G)$.

Up till now, Landau--Ginzburg orbifolds were mostly investigated for the groups $G$ acting diagonally on $\CC^N$ and also for $f$ belonging to a very special class of polynomials~-- the so-called invertible polynomials (cf.\ \cite{BT2,BTW17,BTW16,FJJS,K03,K06,K09}). Also some work was done for the symmetry groups $G = S \ltimes G^d$ with $G^d$ acting diagonally and $S \subseteq S_N$~-- a subgroup of a symmetric group on $N$ elements \cite{BI21,BI22,CJMPW23,EGZ20,EGZ18,I23,MO70,WWP,Y16}.
We relax both conditions in this paper.

Once again, the mirror symmetry attempts to match A- and B-state spaces for {\itshape dual} pairs of Landau--Ginzburg orbifolds. It should be mentioned, that the state space enjoys several other structures besides being just bigraded vector space, like multiplication or bilinear form on both A- and B-sides, which should also be compatible with the mirror map. They will not be considered in the present paper.

\subsection{Calabi--Yau/Landau--Ginzburg correspondence} The final piece of matching comes from Calabi--Yau/Landau--Ginzburg correspondence, which relates respective A-models and their state spaces coming from Calabi--Yau and Landau--Ginzburg geometries of $(X_f,G/J)$ and $(f,G)$ provided $J\subseteq G$.
Mathematically, up till now this was proved for diagonal symmetry groups by Chiodo and Ruan \cite[Theorem 14]{CR11} and in some special cases with $N=5$ in~\cite{Muk}.

However, if this correspondence holds, the vector space~$\B(f,G)$ is also expected to form a~diamond. Namely, it should satisfy the properties (1)--(6) above. Formulated in terms of~$\B(f,G)$ this becomes a purely singularity theoretic question.
It is the main topic of our paper.

\begin{Theorem}\label{theorem: main}
	Let $f \in \CC[x_1,\dots,x_N]$ be a quasihomogeneous polynomial satisfying Calabi--Yau condition and defining an isolated singularity.
	Then for any $G \subseteq \SL_f$, such that $J \subseteq G$ the state space $\B(f,G)$ forms a diamond of size $N-2$ in a sense that it satisfies conditions $(1)$--$(6)$ above.\looseness=-1
\end{Theorem}
\begin{proof}
	The proof is summed up in Propositions~\ref{proposition: diamond under CY}, \ref{proposition: Serre and Hodge under CY} and \ref{proposition: hd0}.
	The vertical and horizontal isomorphism are given in Section~\ref{section: total space}.
\end{proof}

\section{Preliminaries and notations}

\subsection{Quasihomogeneous singularities}

The polynomial $f \in \CC[x_1,\dots,x_N]$ is called \textit{quasihomogeneous} if there are positive integers $d_0,d_1,\dots,d_N$, such that
\begin{equation}\label{QH1}
	f\bigl(\lambda^{d_1}x_1,\dots,\lambda^{d_N}x_N\bigr) = \lambda^{d_0} f(x_1,\dots,x_N), \qquad \forall \lambda \in \CC.
\end{equation}
In what follows we will say that $f$ is quasihomogeneous with respect to the \textit{weights} $d_0,d_1,\dots,d_N$ or the \textit{reduced weights} $q_1 := d_1/d_0,\dots,q_N:= d_N/d_0$.

We will say that $f$ defines an \textit{isolated singularity} at $0 \in \CC^N$ if $0$ is the only critical point of~$f$.
According to K.~Saito \cite[Satz~1.3]{S71} one may consider without changing the singularity only the quasihomogeneous polynomials, such that $0 < q_k \le 1/2$ for all $k=1,\dots,N$. Moreover, we may assume that all variable $x_k$ with $q_k=1/2$ enter $f$ only in monomial $x_k^2$, in particular, there are no monomials of the form $x_ix_j$ with $i\ne j$. Then the number of its monomials is not less than~$N$, the number of variables.

\begin{Example}
 Fermat, chain and loop type polynomials are examples of quasihomogeneous singularities for any natural $a_i \ge 2$
 \begin{align*}
	 f &= x_1^{a_1} \qquad \text{Fermat type},
	 \\
	 f &= x_1^{a_1} + x_1 x_2^{a_2} + \dots + x_{N-1} x_N^{a_N} \qquad \text{chain type},
	 \\
	 f &= x_1^{a_1}x_2 + x_2^{a_2}x_3 + \dots + x_{N-1}^{a_{N-1}}x_N + x_N^{a_N}x_1 \qquad \text{loop type}.
 \end{align*}
 Using the word `type', we mean the certain structure of the monomials without specifying the exponents $a_i$.
\end{Example}

It's easy to see that for Fermat, chain and loop type polynomials, the reduced weights $q_1,\dots,q_N$ are defined in a unique way. This is also true for any quasihomogeneous singularity that we assume (cf.\ \cite[Korollar~1.7]{S71}). We have
\begin{align*}
\begin{split}
 & q_1 = \frac{1}{a_1} \qquad \text{Fermat type,}
 \\
 & q_i = \sum_{j=i}^N\frac{(-1)^{j-i}}{a_1\cdots a_j} \qquad \text{chain type,}
 \\
 & q_i = (-1)^{N-1}\frac{1 - a_i + a_i \sum_{k=2}^{N-1} (-1)^k \prod_{l=2}^k a_{i-l+1}}{\prod_{k=1}^N a_k - (-1)^N} \qquad \text{loop type,}
 \end{split}
\end{align*}
where one assumes $a_0 := a_N$, $a_{-1} := a_{N-1}$, $a_{-2} := a_{N-2}$ and so on.

Given $f \in \CC[x_1,\dots,x_N]$ and $g \in \CC[y_1,\dots,y_M]$ both defining quasihomogeneous singularities it follows immediately that $f+g \in \CC[x_1,\dots,x_N,y_1,\dots,y_M]$ defines a quasihomogeneous singularity too. We will denote such a sum by $f \oplus g$.

\begin{Example}[{\cite[Section 13.1]{AGV85}}]
If $N \le 2$, then all quasihomogeneous isolated singularities are given by the $\oplus$-sums of the Fermat, chain or loop type polynomials.
\end{Example}

\begin{Example}[{\cite[Section 13.2]{AGV85}}]\label{ex: N=3 quasihomogeneous}
If $N = 3$, then all quasihomogeneous isolated singularities are given by the following polynomials $f_{\rm I} = x_1^{a_1} + x_2^{a_2}+ x_3^{a_3}$, $f_{\rm II} = x_1^{a_1} + x_2^{a_2}x_3 + x_3^{a_3}$, $f_{\rm III} = x_1^{a_1} + x_2^{a_2}x_1+ x_3^{a_3}x_1 + \eps x_2^{p}x_3^q$, $f_{\rm IV} = x_1^{a_1} + x_2^{a_2}x_3 + x_3^{a_3}x_2$,
	$f_{\rm V} = x_1^{a_1} + x_2^{a_2}x_1 + x_3^{a_3}x_2$, $f_{\rm VI} = x_1^{a_1}x_2 + x_2^{a_2}x_1 + x_3^{a_3}x_1 + \eps x_2^p x_3^q$, $f_{\rm VII} = x_1^{a_1}x_2 + x_2^{a_2}x_3 + x_3^{a_3}x_1$ with some positive $a_1$, $a_2$, $a_3$. The numbers $a_i$ are arbitrary for $f_{\rm I}$, $f_{\rm II}$, $f_{\rm IV}$, $f_{\rm V}$, $f_{\rm VII}$, however the polynomials $f_{\rm III}$ and $f_{\rm VI}$ are only quasihomogeneous if $\eps \neq 0$ and some additional combinatorial condition on $a_1$, $a_2$, $a_3$ holds. In particular the least common divisor of $(a_2,a_3)$ should be divisible by $a_1-1$ for $f_{\rm III}$ to exist. Allowed values of $\eps$ depend quiet heavily on $a_i$. In particular $\eps \not\in\lbrace 0,\sqrt{-1},-\sqrt{-1}\rbrace $ for $f = x_1^3 + x_1x_2^2 + x_1x_3^2 + \eps x_1x_2^2$ to define an isolated singularity.
\end{Example}

\subsection{Graph of a quasihomogeneous singularity}
Let $f \in \CC[x_1,\dots,x_N]$ define an isolated singularity. Then for every index $j \le N$ the polynomial $f$ has either the summand $x_j^{a}$ or a summand $x_j^ax_k$ for some exponent $a \ge 2$ and index $k \le N$ (cf.\ \cite[Korollar~1.6]{S71} and \cite[Theorem~2.2]{HK}). Construct a map $\kappa\colon \lbrace 1,\dots,N\rbrace \to \lbrace 1,\dots,N\rbrace$. Set $\kappa(j):=j$ in the first case above and $\kappa(j) := k$ in the second.

Following \cite[Section 3]{HK} associate to $f$ the graph\footnote{Such graphs were first considered by Arnold, however with the self-pointing arrows $j \to j$ too. We decide to remove such arrows to reduce complexity.}
$\Gamma_f$ with $N$ vertices labelled with the numbers $1,\dots,N$ and the oriented arrows $j \to \kappa(j)$ if $j \neq \kappa(j)$. In other words, the vertices correspond to the variables $x_i$ and the arrows to the monomials $x_j^ax_k$.

\begin{Example}
The graphs of the $N=3$ quasihomogeneous singularities are all listed in Figure~\ref{fig: N=3 graphs}.
\end{Example}

\begin{figure}[t]
	\centering
\adjustbox{scale=0.8,center}{
 \begin{tikzcd}[column sep=small]
	& \bullet{1} &&&& \bullet{1} &&&& \bullet{2} &&&& \bullet{2} \\
	\\
	\bullet{2} && \bullet{3} && \bullet{2} && \bullet{3} && \bullet{1} && \bullet{3} && \bullet{1} && \bullet{3} \\
	\\
	& \bullet{3} &&&& \bullet{3} &&&& \bullet{1} \\
	\\
	\bullet{1} && \bullet{2} && \bullet{1} && \bullet{2} && \bullet{2} && \bullet{3}
	\arrow[from=1-6, to=3-7]
	\arrow[from=3-11, to=3-9]
	\arrow[from=1-10, to=3-9]
	\arrow[curve={height=6pt}, from=1-14, to=3-15]
	\arrow[curve={height=6pt}, from=3-15, to=1-14]
	\arrow[from=5-2, to=7-3]
	\arrow[from=5-6, to=7-5]
	\arrow[curve={height=6pt}, from=7-5, to=7-7]
	\arrow[curve={height=6pt}, from=7-7, to=7-5]
	\arrow[from=7-3, to=7-1]
	\arrow[from=5-10, to=7-9]
	\arrow[from=7-9, to=7-11]
	\arrow[from=7-11, to=5-10]
	\arrow[draw=none, from=3-1, to=3-3]
	\arrow[draw=none, from=3-5, to=3-7]
	\arrow[draw=none, from=3-13, to=3-15]
	\arrow[shift right=3, draw=none, from=7-5, to=7-7]
\end{tikzcd}
}
 \caption{Graphs of $N=3$ quasihomogenous singularities (see Example~\ref{ex: N=3 quasihomogeneous}).}
 \label{fig: N=3 graphs}
\end{figure}
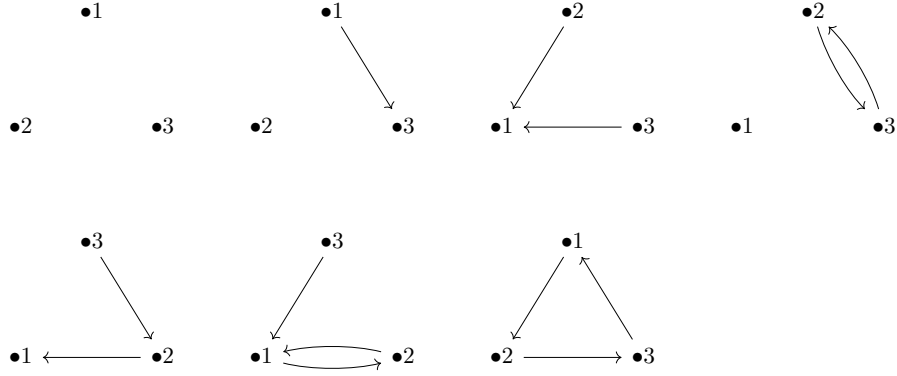

Call a tree oriented if its root has only incoming edges adjacent to it and any other vertex has exactly one outgoing edge and several incoming edges adjacent to it.
The following proposition is immediate.~
\begin{Proposition}[{cf.\ \cite[Lemma 3.1]{HK}}]\label{prop: graph of f}
	Any graph $\Gamma_f$ is a disjoint union of the graphs of the following two types
	\begin{enumerate}\itemsep=0pt
	 \item[$(1)$] oriented tree,
	 \item[$(2)$] oriented circle with the oriented trees having the roots on this oriented circle.
	\end{enumerate}
\end{Proposition}

In what follows we consider the root of the type (1) graph above as a cycle with one vertex. This merges the two types above.

\begin{figure}[t]	\centering
 \[\begin{tikzcd}[sep=tiny]
	&&&&&&&& \bullet & \textcolor{white}{\bullet} & \bullet \\
	&&&&&&&& \bullet && \bullet \\
	&&& \bullet \\
	&& \bullet & \textcolor{white}{\bullet} & \bullet &&&&&&&&&& \bullet & \textcolor{white}{\bullet} \\
	&&& \bullet &&&&& \bullet && \bullet &&&&& \bullet \\
	&&&&&& \bullet && \bullet && \bullet && \bullet \\
	&&&&& \bullet && \bullet && \bullet && \bullet && \bullet \\
	&&&&&& \bullet & \bullet &&&& \bullet & \bullet &&&& \bullet \\
	\bullet & \bullet &&& \bullet & \bullet &&&&&&&&&&& \bullet \\
	\textcolor{white}{\bullet} &&&&&& \bullet &&&&&& \bullet &&&& \textcolor{white}{\bullet} \\
	\bullet & \bullet &&& \bullet & \bullet &&&&&&&&&&& \bullet \\
	&&&&&& \bullet & \bullet &&&&&&&&& \bullet \\
	&&&&& \bullet && \bullet && \bullet \\
	&&&&&& \bullet \\
	&&& \bullet \\
	&& \bullet && \bullet \\
	&& \textcolor{white}{\bullet} & \bullet &&&& \bullet & \bullet & \textcolor{white}{\bullet} & \bullet & \bullet \\
	\\
	&&&&&&&&&& {}
	\arrow[curve={height=6pt}, from=8-12, to=7-10]
	\arrow[curve={height=6pt}, from=7-10, to=8-8]
	\arrow[curve={height=6pt}, from=8-8, to=10-7]
	\arrow[curve={height=6pt}, from=10-7, to=12-8]
	\arrow[curve={height=6pt}, from=12-8, to=13-10]
	\arrow[curve={height=18pt}, dashed, from=13-10, to=10-13]
	\arrow[curve={height=6pt}, from=10-13, to=8-12]
	\arrow[dashed, from=4-16, to=8-12]
	\arrow[from=6-13, to=7-12]
	\arrow[from=7-12, to=8-12]
	\arrow[from=8-13, to=8-12]
	\arrow[from=7-14, to=8-13]
	\arrow[dashed, from=4-15, to=6-13]
	\arrow[dashed, from=5-16, to=7-14]
	\arrow[from=6-9, to=7-10]
	\arrow[from=6-11, to=7-10]
	\arrow[dashed, from=4-4, to=8-8]
	\arrow[from=8-7, to=8-8]
	\arrow[from=7-6, to=8-7]
	\arrow[from=4-3, to=5-4]
	\arrow[from=3-4, to=4-5]
	\arrow[dashed, from=4-5, to=6-7]
	\arrow[from=6-7, to=7-8]
	\arrow[from=7-8, to=8-8]
	\arrow[from=1-9, to=2-9]
	\arrow[from=1-11, to=2-11]
	\arrow[dashed, from=1-10, to=7-10]
	\arrow[dashed, from=2-11, to=5-11]
	\arrow[from=5-11, to=6-11]
	\arrow[dashed, from=2-9, to=5-9]
	\arrow[from=5-9, to=6-9]
	\arrow[dashed, from=5-4, to=7-6]
	\arrow[dashed, from=11-2, to=11-5]
	\arrow[from=11-1, to=11-2]
	\arrow[from=9-1, to=9-2]
	\arrow[dashed, from=9-2, to=9-5]
	\arrow[from=9-5, to=9-6]
	\arrow[from=9-6, to=10-7]
	\arrow[from=11-6, to=10-7]
	\arrow[from=11-5, to=11-6]
	\arrow[dashed, from=10-1, to=10-7]
	\arrow[from=12-7, to=12-8]
	\arrow[from=13-8, to=12-8]
	\arrow[from=13-6, to=12-7]
	\arrow[from=14-7, to=13-8]
	\arrow[dashed, from=16-5, to=14-7]
	\arrow[dashed, from=15-4, to=13-6]
	\arrow[from=17-4, to=16-5]
	\arrow[from=16-3, to=15-4]
	\arrow[dashed, from=17-3, to=12-8]
	\arrow[dashed, from=17-8, to=13-10]
	\arrow[dashed, from=17-9, to=13-10]
	\arrow[dashed, from=17-10, to=13-10]
	\arrow[dashed, from=17-11, to=13-10]
	\arrow[dashed, from=17-12, to=13-10]
	\arrow[dashed, from=8-17, to=10-13]
	\arrow[dashed, from=9-17, to=10-13]
	\arrow[dashed, from=11-17, to=10-13]
	\arrow[dashed, from=12-17, to=10-13]
	\arrow[dashed, from=10-17, to=10-13]
\end{tikzcd}\]
 \caption{Connected component of $\Gamma_f$ (see Proposition~\ref{prop: graph of f}).} \label{fig: graph of f}
\end{figure}
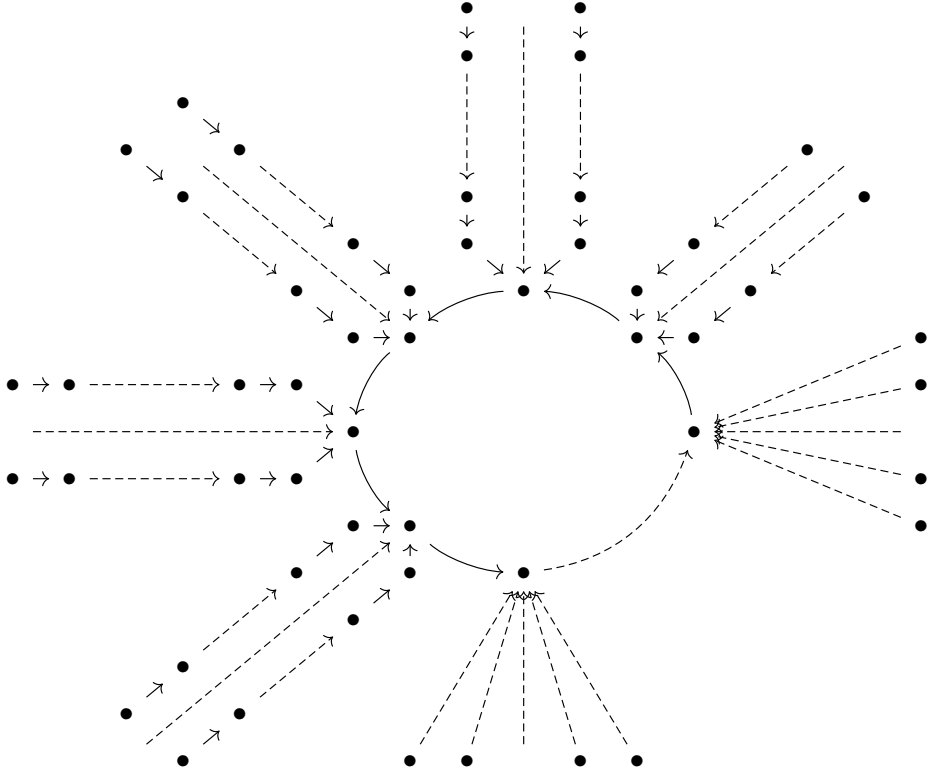

It's easy to see that $\Gamma_{f \oplus g} = \Gamma_f \sqcup \Gamma_g$, but it is not true that $f$ decomposes into the $\oplus$-sum if~$\Gamma_f$ has more than one component. For example the graph of $f = x_1^3 + x_2^3 + x_3^3 + x_1x_2x_3$ is just the disjoint union of three vertices without any edges.

\subsection{Graph decomposition of a polynomial}\label{section: graph decomposition}
Assume we only know the graph $\Gamma_f$ and not the polynomial $f$ itself. The graph structure indicates some monomials that are the summands of $f$.
Call these monomials \textit{graph monomials}. In particular, $f$ has only graph monomials if it is of Fermat, chain or loop type or a $\oplus$-sum of them.

Let $f$ be such that $\Gamma_f$ has only one connected component. Then $\Gamma_f$ has one oriented circle, and $p$ oriented trees with the roots on this circle.
We have the decomposition
\begin{equation}\label{eq: f decomposition}
	f = f_0 + f_1 + \dots + f_p + f_{\rm add},
\end{equation}
with
\begin{enumerate}\itemsep=0pt
 \item[(1)] $f_0,f_1,\dots,f_p, f_{\rm add} \in \CC[x_1,\dots,x_N]$,
 \item[(2)] $f_0$ having as the summands only those graph monomials of $f$, that build up the oriented circle or the common root,
 \item[(3)] $f_k$ having as the summands only those graph monomials of $f$, that build up the $k$-th oriented tree,
 \item[(4)] $f_{\rm add} := f - f_0 - f_1 -\dots - f_p$ having as the summands all the non-graph monomials of $f$.
\end{enumerate}
This decomposition extends easily to the case of $\Gamma_f$ having several components. Note that we could have had $p=0$, but it should always hold that $f_0 \neq 0$.

Writing the decomposition above we had to order the trees by the index of $f_i$. This ordering is not important in what follows.

\begin{Example}
	The polynomial $f = x_1^3 + x_1 \left(x_2^2+x_3^2+x_4^2\right) + \epsilon x_2 x_3 x_4 $ with $\eps \in \CC \backslash \left\lbrace 0, \pm 2 \sqrt{-1}\right\rbrace$ defines an isolated singularity. It's also quasihomogeneous with $q_1=\dots=q_4=1/3$.
	We have $p = 3$,
	\[
		f_0 = x_1^3, \qquad f_1 = x_1 x_2^2, \qquad f_2 = x_1 x_3^2, \qquad f_3 = x_1 x_4^2, \qquad f_{\rm add} = \eps x_2x_3x_4.
	\]
\end{Example}

\subsection{Graph exponents matrix}
Let $f$ define a quasihomogeneous singularity. We introduce the matrix $E_f$ with the entries in~$\ZZ_{\ge 0}$. It follows from Proposition~\ref{prop: graph of f} that $f$ has exactly $N$ graph monomials. Let every row of $E_f$ correspond to a graph monomial. The components of this row will be $(\alpha_1,\dots,\alpha_N)$ if and only if the corresponding graph monomial is $\eps \cdot x_1^{\alpha_1} \cdots x_N^{\alpha_N}$ for some $\eps \in \CC^\ast$.

The matrix $E_f$ is only defined up to a permutation of the rows. We will call it \textit{graph exponents matrix}.

Let $E_{ij}$ denote the components of $E_f$. Then for some non-zero constants $c_k$ we have
\begin{equation}\label{eq: poly from graph exponents matrix}
 	f - f_{\rm add} = \sum_{k=1}^N c_k x_1^{E_{k1}} \cdots x_N^{E_{kN}}.
\end{equation}

\begin{Remark}
	Such a matrix was previously defined in the literature only for the invertible polynomials (see Section~\ref{sec: invertible polynomials}). We consider it here in a wider context.
\end{Remark}

\begin{Example}
	The matrices $E_f$ of the Example~\ref{ex: N=3 quasihomogeneous} are
	\begin{alignat*}{4}
		& E_{f_{\rm I}} = \begin{pmatrix}
					a_1 & 0 & 0 \\ 0 & a_2 & 0 \\ 0 & 0 & a_3
		 \end{pmatrix},
		\qquad&&
		E_{f_{\rm II}} = \begin{pmatrix}
			a_1 & 0 & 0 \\ 0 & a_2 & 1 \\ 0 & 0 & a_3
			\end{pmatrix},
		\qquad&&
		E_{f_{\rm III}} = \begin{pmatrix}
			a_1 & 0 & 0 \\ 1 & a_2 & 0 \\ 1 & 0 & a_3
			\end{pmatrix},&
\\&		E_{f_{\rm IV}} = \begin{pmatrix}
			a_1 & 0 & 0 \\ 0 & a_2 & 1 \\ 1 & 0 & a_3
			\end{pmatrix},
		\qquad &&E_{f_{\rm V}} = \begin{pmatrix}
			a_1 & 0 & 0 \\ 1 & a_2 & 0 \\ 0 & 1 & a_3
			\end{pmatrix},
		\qquad&&
		E_{f_{\rm VI}} = \begin{pmatrix}
			a_1 & 1 & 0 \\ 1 & a_2 & 0 \\ 1 & 0 & a_3
			\end{pmatrix},&
		\\&
		E_{f_{\rm VII}} = \begin{pmatrix}
			a_1 & 1 & 0 \\ 0 & a_2 & 1 \\ 1 & 0 & a_3
			\end{pmatrix}.&&&&&
	\end{alignat*}
\end{Example}

The graph exponents matrices of loop and chain type polynomials read
\begin{equation}\label{eq: chain and loop matrices}
	E_{\rm loop} =
	\left(
		\begin{matrix}
		 a_1 & 1 & \dots & 0 & 0
		 \\
		 0 & a_2 & 1 & \dots & 0
		 \\
		 \vdots & & \ddots & \ddots & 0
		 \\
		 0 & 0 & \dots & a_{N-1} & 1
		 \\
		 1 & 0 & 0 & \dots & a_N
		\end{matrix}
	\right),
	\qquad
	E_{\rm chain} =
	\left(
		\begin{matrix}
		 a_1 & 0 & \dots & 0 & 0
		 \\
		 1 & a_2 & 0 & \dots & 0
		 \\
		 \vdots & \ddots & \ddots & \ddots & \vdots
		 \\
		 0 & \dots & 1 & a_{N-1} & 0
		 \\
		 0 & 0 & \dots & 1 & a_N
		\end{matrix}
	\right).
\end{equation}

In general, by Proposition~\ref{prop: graph of f} if $\Gamma_f$ has only one connected component, the matrix $E_f$ after some renumbering of the variables and the rows has the block form. The diagonal blocks are several chain type exponent matrices and exactly one loop type exponents matrix as in equation~\eqref{eq: chain and loop matrices}, such that for every chain type block there is exactly one additional matrix entry~$1$ in the first row of this block and the column of a loop type block. All the other matrix entries except listed vanish,
\begin{equation}\label{eq: Ef special form}
	E_f =
	\left(
		\begin{array}{cccc}
			\multicolumn{1}{c|}{A_0} & 0 & 0 & 0
			\\
			\hhline{--|}
			\multicolumn{1}{c|}{U_{i_1j_1}} & \multicolumn{1}{c|}{A_1} & 0 & 0
			\\
			\hhline{--~|}
			\vdots & & \ddots & 0
			\\
			\hhline{-~~-|}
			\multicolumn{1}{c|}{U_{i_pj_p}} & 0 & 0 & \multicolumn{1}{|c|}{A_p}
		\end{array}\,
	\right),
\end{equation}
where $A_0$ is a loop type polynomial exponents matrix and $A_1,\dots,A_p$ are chain type polynomial exponent matrices, the matrix $U_{ij}$ is the rectangular matrix with $1$ at position $(i,j)$ and all other entries $0$.

Assuming the decomposition of equation~\eqref{eq: f decomposition}, the matrix $A_0$ is exactly the exponent matrix of $f_0$ and the matrices $A_1,\dots,A_p$ are defined by $f_1,\dots,f_p$.

\begin{Proposition}\label{prop: graph exponent matrix is invertible}
	Let $f$ define a quasihomogeneous singularity. Then
	\begin{enumerate}\itemsep=0pt
	 \item[$(i)$] the matrix $E_f$ is invertible,
	 \item[$(ii)$] there is a canonical choice of the integer weights $(d_0,d_1,\dots,d_N)$.
	\end{enumerate}
\end{Proposition}
\begin{proof}
	Let $f$ be quasihomogeneous with the reduced weights $(q_1,\dots,q_N)$. We show first that these weights are defined uniquely by the graph monomials of $f$.

 Let $f$ be decomposed as in equation~\eqref{eq: f decomposition}. Then $f_0$ is of Fermat or loop type and the weights of its variables are defined uniquely. Similarly for any $f_k$ with $k=1,\dots,p$ corresponding to the tree with the root on the oriented circle, the weight of the root's variable is defined by the quasihomogeneity of $f_0$, going up the tree of deduces uniquely the weight of every variable of $f_k$ corresponding to the consequent vertex.

 Introduce two $\ZZ^N$ vectors: $\overline{q} := (q_1,\dots,q_N)^{\mathsf{T}}$ and $\boldsymbol{1} := (1,\dots,1)^{\mathsf{T}}$. Then the quasihomogeneity condition on $f$ is equivalent to the $\QQ^N$ vector equality $E_f \cdot \overline{q} = \boldsymbol{1}$. It follows now from Cramer rule that $\det(E_f) \neq 0$ because this equation has a unique solution. This completes (i).

	The canonical weight set is obtained by taking $d_f := \det(E_f)$ and solving $E_f \cdot \overline{d} = d_f \boldsymbol{1}$ for~$\overline{d} := (d_1,\dots,d_N)^{\mathsf{T}}$.
\end{proof}

\subsection{Invertible polynomials}\label{sec: invertible polynomials}

The set of all quasihomogeneous singularities contains the following important class. The polynomial $f$ defining an isolated quasihomogeneous singularity having no monomial of the form~$x_ix_j$ with $i\ne j$ and as many monomials as the variables is called \textit{invertible polynomial} and is said to define an \textit{invertible singularity}.

The following statement can be assumed as a well-known (cf.~\cite{KS}), we add up the proof for completeness.
\begin{Proposition}\label{prop: invertible polynomial classification}
	Let $f$ be an invertible polynomial. Then after some rescaling and renumbering of the variables we have $f = f^{(1)} \oplus \cdots \oplus f^{(n)}$ for $f^{(k)}$ being either of Fermat, chain or loop type.
\end{Proposition}
\begin{proof}
Assume $\Gamma_f$ to contain a vertex with two incoming arrows. Then $f$ is of the form
\[
	\alpha_1 x_i^a x_l^K + \alpha_2 x_j^bx_i + \alpha_3 x_k^cx_i + g(x),
\]
where $K \in \{ 0,1\}$, $\alpha_1\alpha_2\alpha_3 \neq 0$, $b,c \ge 2$ and $g$ does not depend on $x_k$, $x_i$, $x_j$. Computing
\begin{align*}
	\frac{\p f}{\p x_i}  = a \alpha_1 x_i^{a-1} x_l^K + \alpha_2 x_j^b + \alpha_3 x_k^c,
	\qquad
	\frac{\p f}{\p x_j}  = b \alpha_2 x_j^{b-1}x_i,
	\qquad
	\frac{\p f}{\p x_k} = a \alpha_3 x_k^{c-1}x_i.
\end{align*}
Setting $x_i = 0$, we see that vanishing $\frac{\p f}{\p x_i} = \frac{\p f}{\p x_j} =\frac{\p f}{\p x_k} =0$ is equivalent to $\alpha_2 x_j^b + \alpha_3 x_k^c = 0$ what shows that $x_i=x_j=x_k=0$ is not an isolated critical point of $f$.
\end{proof}

The graphs of invertible singularities are disjoint unions of isolated vertices (Fermat types), oriented cycles (loop types) and one branch trees (chain types).

In the notation of Section~\ref{section: graph decomposition}, the graph decomposition equation~\eqref{eq: f decomposition} of Fermat, loop and chain type polynomials is the following. We have always $f_{\rm add} = 0$, $p=0$ and $f_0 = f$ for Fermat and loop types, but $p=1$ and $f_0 + f_1 = f$ for chain type with $f_0 = x_m^{a_m}$.

\begin{Example}
 The quasihomogeneous singularities with $N=2$ are all invertible.
	The quasihomogeneous singularities with $N=3$ are not all invertible. In the notation of Example~\ref{ex: N=3 quasihomogeneous}, we have $f_{\rm I}$~-- Fermat$\oplus$Fermat$\oplus$Fermat, $f_{\rm II}$~-- Fermat$\oplus$chain, $f_{\rm III}$~-- not invertible, $f_{\rm IV} $~-- Fermat$\oplus$loop,
	$f_{\rm V}$~-- chain, $f_{\rm VI}$~-- not invertible, $f_{\rm VII}$~-- loop.
\end{Example}

\section{Symmetries}\label{section: symmetries}
Given a quasihomogeneous polynomial $f = f(x_1,\dots,x_N)$ consider the
\textit{maximal group of linear symmetries} of $f$ defined by
\[
	\GL_f :=  \lbrace g \in \GL(N,\CC) \mid f(g \cdot \bx) = f(\bx) \rbrace.
\]

\begin{Lemma}
 Under our assumptions on $f$ any $g \in \GL_f$ necessarily preserves the weights of the variables, i.e., maps each $x_i$ to a linear combination of $x_j$ with the same weight.
\end{Lemma}
\begin{proof}
The action of $g$ preserves the homogeneous components of $f$. In particular, the variables in the quadratic terms of $f$ map to linear combinations of variables in the quadratic terms of $f$ and hence weight $1/2$ subspace is preserved by $f$.

We may now assume that $f$ has no quadratic terms. In this case, the argument of \cite[Theorem 2.1]{OS78} applies verbatim to the quasihomogeneous situation to prove that $\GL_f$ is finite. Let \[\wt{\GL}_f= \lbrace g \in \GL(N,\CC) \mid f(g \cdot \bx) = \chi(g)f(\bx), \chi(g)\in\CC^*  \rbrace.\]
The map $\chi\colon\wt{\GL}_f\to\CC^*$ is a character and $\GL_f$ is precisely the kernel of $\chi$, in particular, it is a normal subgroup. Moreover, the condition \eqref{QH1} provides an inclusion \smash{$t\colon\CC^*\hookrightarrow\wt{\GL}_f$}, such that $\chi\circ t$ is a degree $d_0>0$ map. The action by conjugation of connected subgroup $t(\CC^*)$ on the finite subgroup $\GL_f$ is necessary trivial, which means that $\GL_f$ commute with $t(\CC^*)$. This means that $\GL_f$ preserves the eigenspaces of $t(\CC^*)$ as desired.
\end{proof}

\begin{Remark}
 Notably, the observation of this Lemma seems to be unknown to previous authors (for example, it was imposed as condition in \cite{Muk} and some subsequent works).
\end{Remark}

Let $G_f^d \subseteq \GL_f$ be the \textit{maximal group of diagonal symmetries} of $f$. This is the group of all diagonal elements of $\GL(N,\CC)$ belonging to $\GL_f$.

We have
\[
	G_f^d \cong \bigl\lbrace (\lambda_1,\dots,\lambda_N) \in (\CC^*)^N \mid f(\lambda_1x_1,\dots,\lambda_Nx_N) = f(x_1,\dots,x_N) \bigr\rbrace.
\]

It's obvious that
\begin{equation}\label{eq: full symmetry group of oplus}
	G_{f' \oplus f''}^d \cong G_{f'}^d \times G_{f''}^d.
\end{equation}
Note, however, that the same does not necessarily hold for $\GL_{f' \oplus f''}$.

In what follows we will use the notation
\[
\epi\left[\alpha\right] := \exp\bigl(2 \pi \sqrt{-1} \alpha\bigr), \qquad \alpha \in \RR.
\]
Each element $g\in G_f^d$ has a unique expression of the form
\begin{equation*}
g={\rm diag}\left(\epi\left[\frac{\alpha_1}{r}\right], \dots,\epi\left[\frac{\alpha_N}{r}\right]\right)
\qquad \mbox{with}\quad 0 \leq \alpha_i < r,\quad \alpha_i\in\NN,
\end{equation*}
where $r$ is the order of $g$. We adopt the additive notation
\[
{\overline g = (\alpha_1/r,\dots, \alpha_N/r)} \qquad \text{ or } \qquad
{\overline g = \frac{1}{r}(\alpha_1,\dots, \alpha_N) }
\]
for such an element $g$.

\begin{Example}
For $f = x_1^{a_1}$	we have $\GL_f = G_f^d = \langle g \rangle $ with $g \in \CC^\ast$ acting by $g(x_1) = \epi\bigl[\frac{1}{a_1}\bigr] \cdot x_1$. Its order is $a_1$ and in the additive notation we have $\overline g = (1/a_1)$, giving us $\GL_f \cong \ZZ / a_1 \ZZ$.
\end{Example}
\begin{Example}\label{example: symmetry group chain2}
For $f = x_1^{a_1}x_2 + x_2^{a_2}$ we have $G_f^d = \langle g_1, g_2 \rangle$ with $g_1 \cdot (x_1,x_2) = (\epi[\frac{1}{a_1}] x_1, x_2)$ and ${ g_2 \cdot (x_1,x_2) = \bigl(\epi[\frac{1}{a_2}(1 - \frac{1}{a_1})]x_1, \epi\bigl[\frac{1}{a_2}\bigr]x_2\bigr)}$. In the additive notation $\overline g_1 = (1/a_1, 0)$ and $\overline g_2 = ((a_1 - 1)/(a_1a_2), 1/a_2)$.
In this example $\GL_f = G_f^d$ because $q_1 \neq q_2$.
\end{Example}

Let $(q_1,\dots,q_N)$ be the reduced weight set of $f$. Then we have
\[
	j_f:= (\epi[q_1],\dots, \epi[q_N]) \in G_f^d.
\]
In particular, it follows that $G_f^d$ and $\GL_f$ are not empty whenever $f$ is quasihomogeneous.
Denote by $J$ the group generated by $j_f$:
\[
	J := \langle j_f \rangle \subseteq G_f^d.
\]
Since $g\in \GL_f$ preserves the weights we see that $j_f$ commutes with $g$. In other words $J$ is the central subgroup of~$\GL_f$.

\subsection[Fixed loci of the GL\_f elements]{Fixed loci of the $\boldsymbol{\GL_f}$ elements}

For each $g\in \GL_f$, denote by $\Fix(g)$ \textit{the fixed locus of $g$}
\begin{equation*}
 	\Fix(g):=\bigl\lbrace (x_1,\dots,x_N) \in\CC^N \mid g \cdot (x_1,\dots,x_N) = (x_1,\dots,x_N) \bigr\rbrace.
\end{equation*}
This is an eigenvalue $1$ subspace of $\CC^N$ and therefore a linear subspace of $\CC^N$.
By $N_g:=\dim_\CC\Fix(g)$ denote its dimension and by $f^g:=f|_{\Fix(g)}$ the restriction of $f$ to the fixed locus of $g$.
For $g\in G_f^d$, this linear subspace is furthermore a span of a collection of standard basis vectors.

For each $h\in G_f^d$, let $I_h := \{i_1,\dots,i_{N_h}\}$ be a subset of $\{1,\dots, N\}$ such that
\[
\Fix(h)=\bigl\{(x_1,\dots,x_N)\in\CC^N\mid x_{j}=0, j\notin I_h\bigr\}.
\]
In the other words, $\Fix(h)$ is indexed by $I_h$.
In particular, $I_{\id}=\{1,\dots, N\}$.

More generally, for $g\in \GL_f$, since $g$ preserves the weight subspaces of $\CC^N$, the weights of the subspace $\Fix(g)$ are well-defined and are the subset of $\{q_1,\ldots,q_N\}$. Fix a subset ${I_g\subset \{1,\ldots,N\}}$ such that $q_k$ with $k\in I_g$ are exactly all the weights of $\Fix(g)$, so that, in particular, we have~${|I_g|=N_g}$. Note that if $g\not\in G_f^d$ there is no canonical choice for $I_g$, but the choice made at this step will not impact our results.

Denote by $I_h^c$ the complement of $I_h$ in $I_\id$ and set $d_h:=N-N_h$, the codimension of $\Fix(h)$.

\begin{Proposition}
	For any diagonalizable $g \in \GL_f$ with $N_g > 0$ there is a choice of coordinates on $\Fix(g)$ linear in $x_i$, such that the polynomial $f^g$ also defines a quasihomogeneous singularity.
\end{Proposition}
\begin{proof}
	Let $\widetilde x_1,\dots, \widetilde x_N$ be the coordinates of $\CC^N$ dual to the basis diagonalizing $g$. In this coordinates the polynomial $f^g$ is obtained by setting some of $\widetilde x_\bullet$ to zero. The proof follows now by the same argument as in \cite[Proposition~5]{ET13}.
\end{proof}

\begin{Example}
If $f$ is of Fermat or loop type, then any $g \in \GL_f^d$, such that $g \neq \id$ satisfies $\Fix(g) = 0$.
If $f = x_1^{a_1} + x_1x_2^{a_2} + \dots +x_{N-1}x_N^{a_N}$ is of chain type, then any $g \in \GL_f^d$, such that $g \neq \id$ satisfies
$\Fix(g) = \bigl\lbrace (x_1,\dots,x_p,0,\dots,0) \in \CC^N \mid x_k \in \CC \bigr\rbrace$ for some $p$ depending on~$g$.
The polynomial $f^g$ is of chain type again: $f^g = x_1^{a_1} + x_1x_2^{a_2} + \dots +x_{p-1}x_p^{a_p}$.
\end{Example}

Denote also
\[
\SL_f:=\GL_f\cap\SL(N,\CC).
\]
This group will be important later on because it preserves the volume form of $\CC^N$.

\subsection[Age of a GL\_f element]{Age of a $\boldsymbol{\GL_f}$ element}

For $g\in \GL_f$ let $\lambda_1,\ldots,\lambda_N$ be the collection of its eigenvalues. Let $0\le \alpha_i < 1$ be such that $\lambda_i=\epi[\alpha_i]$, then {\em age} of $g$ is defined as the number
\begin{equation*}
\age(g) := \sum_{k=1}^N \alpha_k.
\end{equation*}

The following properties are clear but will be important in what follows.

\begin{Proposition}\label{prop: integral age}\quad
\begin{itemize}\itemsep=0pt
\item[$(1)$] For any $g \in \GL_f$ we have
\begin{equation}\label{eq: age g and inverse}
		\age(g) + \age\bigl(g^{-1}\bigr) = N - N_g = d_g.
\end{equation}

\item[$(2)$] For a diagonalizable $g\in \GL_f$ we have ${\rm age}(g) = 0$ if and only if $g = \id$.

\item[$(3)$] We have $g\in \SL_f$ if and only if ${\rm age}(g)\in\NN$.
\end{itemize}
\end{Proposition}

\subsection[Diagonal symmetries and a graph Gamma\_f]{Diagonal symmetries and a graph $\boldsymbol{\Gamma_f}$}

\begin{Proposition}\label{prop: group elements on graph}
	Let $\Gamma_f$ be a graph of a quasihomogeneous singularity $f$ and $g \in \GL_f$. Then if $g$ acts nontrivially on $x_k$, then it acts nontrivially on all $x_i$, such that there is an oriented path from $i$-th to the $k$-th vertex.
\end{Proposition}
\begin{proof}
 We first show the statement for the arrows pointing at $k$.
 Having an arrow $j \to k$ means that $f$ has a monomial $x_j^{a_j}x_k$ as a summand with a nonzero coefficient. We have $g \cdot x_k \neq x_k$ and therefore the summand can only be preserved under the action of $g$ if $g\cdot x_j \neq x_j$. Having an oriented path $i \to j_1 \to \cdots \to j_n \to k$ we have by using the previous step that $g \cdot x_{j_n} \neq x_{j_n}$ and then $g \cdot x_{j_a} \neq x_{j_a}$ for all $a$. Hence, for $x_i$ too.
\end{proof}

Let $E_f$ be the graph exponents matrix of $f$. Consider
\begin{equation*}
G_f^{\rm gr} := \Biggl\{(\lambda_1,\dots,\lambda_N)\in (\CC^*)^N \, \Biggl|\, \prod_{j=1}^N \lambda_j ^{E_{1j}}=\dots =\prod_{j=1}^N\lambda_j^{E_{Nj}}=1 \Biggr\} \end{equation*}
with $E_{ij}$ being the components of \smash{$E_f$}.
The group \smash{$G_f^{\rm gr}$} is exactly the maximal group of diagonal symmetries of the difference $f-f_{\rm add}$. In particular, every element of \smash{$G_f^{\rm gr}$} preserves all graph monomials of $f$.

We have \smash{$G_f^d \subseteq G_f^{\rm gr}$} and hence \smash{$G_f^d$} is a finite group.
An element ${\overline g = \frac{1}{r}(\alpha_1,\dots,\alpha_N)}$ belonging to $G_f^{\rm gr}$ satisfies
$E_f \cdot \overline g \in \ZZ^N$.
This gives yet another characterization of the group $G_f^{\rm gr}$
\begin{equation*}
	G_f^{\rm gr} \cong \bigl\lbrace \overline g \in (\QQ/\ZZ)^N \mid E_f \cdot \overline g \in \ZZ^N \bigr\rbrace = E_f^{-1} \ZZ^N / \ZZ^N.
\end{equation*}

It follows that every vector $\overline g$ giving a \smash{$G_f^{\rm gr}$}-element is a linear combination with integer coefficients of the columns of \smash{$E_f^{-1}$}. Following the notation of Krawitz~\cite{K09} define $\overline \rho_i$ as the $i$-th column of $E_f^{-1}$
\begin{equation*}
	E_f^{-1} =  ( \overline \rho_1 | \dots | \overline \rho_N  ).
\end{equation*}
Denote also $\rho_i := \epi[\overline \rho_i] \in G_f^d$.

The elements $\rho_k$ generate $G_f^{\rm gr}$ and $j_f = \rho_1\cdots\rho_N$.
The columns of $E_f$ generate all relations on $\rho_1, \dots, \rho_N$.

In particular, for $(E_{1k}, \dots, E_{Nk})^{\mathsf{T}}$ being a $k$-th column of $E_f$ we have in $G_f^{\rm gr}$
\[
\rho_1^{E_{1k}}\cdots \rho_N^{E_{Nk}} = \id,
\]
and all other relations among $\lbrace \rho_k \rbrace_{k=1}^N$ follow from those written above.

\subsection{Diagonal symmetries of an invertible singularity}

In \cite{FJJS} for an invertible $f$, the authors gave the set $\calS_f$ of all $N$-tuples $(s_1,\dots,s_N)$ such that every $g~\in~G_f^d~\backslash\lbrace\id\rbrace$ is written uniquely by
\[
	g = \prod_{k \in I_g^c} \rho_k^{s_k},
\]
and $s_k = 0$ if and only if $k \in I_g$. Due to equation~\eqref{eq: full symmetry group of oplus} and Proposition~\ref{prop: invertible polynomial classification}, it is enough to construct such set for Fermat, loop or chain type polynomials.

\begin{Proposition}\label{proposition: chain loop group structure}
For $f$ being of Fermat, chain or loop type the set $\calS_f$ consists of all $s = (s_1,\dots,s_N)$, such that
	\begin{itemize}\itemsep=0pt
	 \item $($Fermat type$)$: $1 \le s_1 \le a_1-1$
	 \item $($loop type$)$: $1 \le s_k \le a_k$ and
	 \begin{equation*}\label{eq: vanishing r in loop}
		s \neq (a_1,1,a_3,1,\dots, a_{N-1},1), \qquad s \neq (1,a_2,1,a_4,1,\dots, a_{N})
	 \end{equation*}
	 if $N$ is even.
	 \item $($chain type$)$: $s$ is of the form
	 \begin{equation*}
	 (0,\dots,0,s_p,s_{p+1},\dots,s_{N}), \qquad \text{with} \quad \{1,\dots,p-1\} = I_g,
	 \end{equation*}
	 with $1 \le s_{p} \le a_{p}-1$, $1 \le s_k \le a_k$ for $k > p$.
	\end{itemize}
\end{Proposition}

\begin{Corollary}
	In the additive notation for the column ${s^{\mathsf{T}} = (s_1,\dots,s_N)^{\mathsf{T}}}$, we have ${\overline g = E_f^{-1} s^{\mathsf{T}}}$\!.
\end{Corollary}

\subsection{Diagonal symmetries of a quasihomogeneous singularity}

For any quasihomogeneous singularity $f$, consider its graph decomposition as in equation~\eqref{eq: f decomposition}. Up to the renumbering and rescaling of the variables, we have
\begin{align*}
	& f_0 = x_1^{a_1} \qquad \text{or} \qquad f_0 = x_1^{a_1}x_2 + \dots + x_K^{a_K}x_1,
	\\
	& f_1 = x_1 x_{K+1}^{b_1} + x_{K+1}x_{K+2}^{b_2} + \dots + x_{K+L-1}x_{K+L}^{b_L}
\end{align*}
with the similar expression for $f_2, \dots, f_p$.

Any nontrivial \smash{$g \in G_{f_0}^d$} extends to an element \smash{$\widetilde g \in G_f^{\rm gr}$}. Moreover it follows that $\Fix(g) = 0$ and also $\Fix(\widetilde g) = 0$ as long as $g\ne\id$. Similarly any element \smash{$h \in G_f^{\rm gr}$} with $\Fix(h) = 0$ acts nontrivially on $x_1,\dots,x_K$ preserving $f_0$. Hence it defines \smash{$h_0 \in G_{f_0}^d$} by the restriction.

At the same time any $h \in (\CC^\ast)^{L}$ acting diagonally on $(x_{K+1},\dots,x_{K+L})$ preserving $f_1$ extends to an element of \smash{$G_f^{\rm gr}$} assuming it to act trivially on $f_0$ and all other $f_2,\dots,f_p$. One notes immediately that such elements $h$ are the elements of chain type polynomial symmetry group. Denote the group of all such elements by \smash{$G_{f_1}^\circ$}.

We construct the groups \smash{$G_{f_2}^\circ, \dots, G_{f_p}^\circ$} in a similar way.

For a nontrivial element \smash{$g \in G_f^{\rm gr}$} and its restriction \smash{$g_0 \in G^d_{f_0}$}, the extension $\widetilde g_0$ is not unique. However, having it fixed, we have by Proposition~\ref{prop: group elements on graph} that there is a unique set of \smash{$g_k \in G_{f_k}^\circ$} for $k=1,\dots,p$, s.t.
\[
	g = \widetilde g_0\cdot g_1 \cdots g_p,
\]
We have that every $g_k$ acts non-trivially only on the variables of $f_k$ preserving all the variables of $f_0$ identically.

We have
\[
	|G_f^{\rm gr}| = |G_{f_0}^d|\cdot |G_{f_1}^{\rm gr}| \cdots |G_{f_p}^{\rm gr}|.
\]
Associate to every $g_0,g_1,\dots,g_p$ an element $s_0,s_1,\dots,s_p$ as in Proposition~\ref{proposition: chain loop group structure}. Composing them in one column $s$, we have
\[
	\overline g = E_f^{-1} s^{\mathsf{T}}.
\]
Note that for $s_0 \neq 0$ and $s_1=0, \dots, s_p =0$, all components of $\overline g$ are nonzero. We follow the convention $s_1 \neq 0, \dots, s_p \neq 0$ if $g$ is such that $g_0 \neq \id$.

The following proposition is very important in what follows.

\begin{Proposition}\label{prop: age by matrix}
	For any $g \in G_f^d$ such that $\overline g = E_f^{-1} s$, we have
	\[
	\age(g) = (1,\dots,1) E_f^{-1} s^{\mathsf{T}}.
	\]
\end{Proposition}
\begin{proof}
	We need to show that the components of $\overline g$ belong to $[0,1)$. This follows immediately from the equality $E_f \overline g = s$, the bounds on $s$ and the special form of the matrix $E_f$ (see equation~\eqref{eq: Ef special form}).
\end{proof}

\subsection{Symmetries and the Calabi--Yau condition}\label{section: symmetries under CY}
Let the reduced weight set $q_1,\dots,q_N$ of $f$ satisfy $\sum_{k=1}^N q_k = 1$. This equality is called the \textit{Calabi--Yau} condition and we will say that \textit{$f$ satisfies the CY condition}. We show in this section that it puts significant restrictions on the symmetries of $f$.

Let the matrix $E_f^{\mathsf{T}}$ define a polynomial $f^{\mathsf{T}}$. Namely, if for $f$ we have~\eqref{eq: poly from graph exponents matrix}, then
\[
	f^{\mathsf{T}} := \sum_{k=1}^N c_k x_1^{E_{1k}} \cdots x_N^{E_{Nk}}.
\]
This polynomial does not necessarily define an isolated singularity. However, it is quasihomogeneous again with some weights $q_1^{\mathsf{T}}, \dots, q_N^{\mathsf{T}}$ by the same argument as in Proposition~\ref{prop: graph exponent matrix is invertible}.

We call $f$ star-shaped if its graph $\Gamma_f$ consists of $N-1$ vertices all adjacent to one vertex. Namely, $f_0 = x_1^{a_1}$, $p=N-1$ and \smash{$f_i = x_1 x_{i+1}^{b_{i+1}}$}. Such a polynomial satisfies the CY condition if and only if
\[
 \sum_{i=2}^N \frac{1}{b_i} = 1.
\]
Example of such a polynomial is given by
\[
 f = x_1^{a_1} + x_1 \bigl(x_2^3 +x_3^3 +x_4^3\bigr) + x_2^2x_3^2+ x_2^2x_4^2+ x_3^2x_4^2
\]
with the Milnor number $81$.

We will treat the star-shaped polynomials separately.

\begin{Proposition}
 Let $f$ not being a star-shaped polynomial, satisfy the CY condition. Then the weights $q_1^{\mathsf{T}}, \dots, q_N^{\mathsf{T}}$ are all positive.
\end{Proposition}
\begin{proof}
 This lemma is obvious for invertible polynomial $f$ and we assume only noninvertible cases in the proof.

 Let $E_f$ be written as in equation~\eqref{eq: Ef special form} and $A_0$ be a $K \times K$ loop type matrix as in equation~\eqref{eq: chain and loop matrices}. It is immediate that $q_{K+1}^{\mathsf{T}},\dots,q_N^{\mathsf{T}}$ are positive.

 For $i=1,\dots,K$, denote by $\A_i$ the sum of all $q_j^{\mathsf{T}}$ with $j > K$, s.t.\ $j$-vertex of $\Gamma_f$ is adjacent to the $i$-th vertex.

 \begin{Lemma}
 We have $0 \le \A_i < 1$ for any $1 \le i \le K$.
 \end{Lemma}
 \begin{proof}
 $\A_i$ is non-negative as the sum of the positive weights. However this sum can be empty.

 Let $\A_i \ge 1$ for some $i$. Let the vertices adjacent to the $i$-th vertex be labelled by $K+1,\dots,K+m$ contributing to $f$ with the monomials \smash{$x_ix_{K+1}^{b_{K+1}}, \dots, x_ix_{K+m}^{b_{K+m}}$}.
 Then
 \[
 q_{K+j} = \frac{1}{b_{K+j}}(1 - q_i) \qquad \text{and} \qquad q_{K+j}^{\mathsf{T}} \le \frac{1}{b_{K+j}}.
 \]
 Denote $S := \sum_{j=1}^m \dfrac{1}{b_{K+j}}$. If the CY condition holds, then
 \[
 q_i + S - q_i S \le 1 \ \Leftrightarrow \ (S - 1) \le q_i(S-1).
 \]
 If $S > 1$, this gives $q_i \ge 1$ which contradicts the quasihomogeneity condition of $f$. If $S=1$, then $q_{K+1} + \dots + q_{K+m} = 1 - q_i$ and the CY condition can only hold if $f$ is a star-shaped CY polynomial.
 \end{proof}

 Let us show that $q_1^{\mathsf{T}}$ is positive. The proof for $q_2^{\mathsf{T}},\dots,q_N^{\mathsf{T}}$ is similar.

 Let $c_{ij}$ stand for the components of the $K \times K$ matrix $A_0^{-1}$. Note that up to a sign these are just the products of $a_i$ divided by $\det A = a_1\cdots a_K + (-1)^{K-1}$.
 In particular, we have
 \begin{align*}
 &c_{i,i} = \frac{a_1 \cdots a_K}{a_i \det A}, \qquad c_{i,i+r} = (-1)^r \frac{a_1 \cdots a_K}{a_i \cdots a_{i+r} \det A}, \qquad 1 \le r \le K-i,
 \\
 &c_{i,i-1} = (-1)^{K-1}\frac{1}{\det A}, \qquad c_{i,i-r} = (-1)^{K-r}\frac{a_{i-r+1}\cdots a_{i-1}}{\det A}, \qquad 1 \le r \le i-1.
 \end{align*}
 Then
 \[
 q_i^{\mathsf{T}} = c_{1i} + \dots + c_{Ki} - c_{1i} \A_1 - \dots - c_{Ki} \A_K
 \]
 for $1 \le i \le K$ and
 \begin{align*}
 & \sum_{i=1}^N q_i \ge \sum_{i=1}^K q_i + \sum_{i=1}^K \A_i (1 - q_i)
 = \sum_{i,j=1}^K c_{ij} + \sum_{i=1}^K \A_i (1 - c_{i1} - \dots - c_{iK}).
 \end{align*}
 Under the CY condition we have
 \[
 \sum_{i,j=1}^K c_{ij} + \sum_{i=2}^K \A_i (1 - c_{i1} - \dots - c_{iK}) - 1 \le - \A_1 (1 - c_{11} - \dots - c_{1K}).
 \]
 The bracket on the right hand side is positive because $q_1 < 1$. This gives the estimate
 \[
 - c_{11}\A_1 \ge
 \frac{c_{11}}{1 - c_{11} - \dots - c_{1K}} \left( \sum_{i,j=1}^K c_{ij} -1 + \sum_{i=2}^K \A_i (1 - c_{i1} - \dots - c_{iK}) \right)
 \]
 because $c_{11}$ is positive. We get then the estimate
 \begin{align*}
 q_1^{\mathsf{T}} \ge{}& c_{11} + \dots + c_{K1} + \frac{c_{11}}{1 - c_{11} - \dots - c_{1K}} \left( \sum_{i,j=1}^K c_{ij} - 1 \right)
 \\
 & + \sum_{i=2}^K \A_i \left( - c_{i1} + c_{11} \frac{1 - c_{i1} - \dots - c_{iK}}{1 - c_{11} - \dots - c_{1K}} \right)
 \\
 ={}& \frac{(c_{11}+\dots + c_{K1}) (1 - c_{11} - \dots - c_{1K}) + c_{11} \bigl( \sum_{ij=1}^K c_{ij} -1\bigr)}{1 - c_{11} - \dots - c_{1K}}
 \\
 & + \sum_{i=2}^K \A_i \frac{c_{11}(1 - c_{i1} - \dots - c_{iK}) - c_{i1} (1 - c_{11} - \dots - c_{1K})}{1 - c_{11} - \dots - c_{1K}}.
 \end{align*}
 Introduce the positive numbers $T_r$ and $P_r$ by
 \begin{align*}
 &T_r = a_{K-1}(\cdots a_{r-2}(a_{r-1}(a_r-1) + 1) -1) \cdots + (-1)^{r-1},
 \\
 &P_{K-1} = a_{K-1}, \qquad P_r = a_r (T_{r+2} + P_{r+1} + (-1)^{r}), \qquad r \le K-2.
 \end{align*}
 Some computations give us
 \begin{align*}
 & (c_{11}+\dots + c_{K1}) (1 - c_{11} - \dots - c_{1K}) + c_{11} \left( \sum_{ij=1}^K c_{ij} -1\right) = \frac{a_K}{\det A} (T_3 + P_2 -1)
 \end{align*}
 and
 \begin{align*}
 &(1 - c_{21} - \dots - c_{2K}) - c_{21} (1 - c_{11} - \dots - c_{1K}) = \frac{a_K}{\det A} (T_2 + 1),
 \\
 &(1 - c_{i1} - \dots - c_{iK}) - c_{i1} (1 - c_{11} - \dots - c_{1K}) = \frac{a_K}{\det A} \left(T_i + (-1)^i\right) \prod_{r=2}^{i-1} a_i, \qquad i \ge 3.
 \end{align*}
 These are positive numbers for $a_i \ge 2$, what gives the proof after applying the lemma above.
\end{proof}

\begin{Proposition}\label{prop: age greater than weights}
 Let $f$ satisfy the CY condition.
	Then for any diagonalizable $g \in \GL_f$ such that $N_g = 0$, we have
	\begin{equation*}
		\age(g) \ge \sum_{k=1}^N q_k.
	\end{equation*}

	The equality is only reached if $g = j_f$.
\end{Proposition}
\begin{proof}
Rewrite $f$ in the coordinates $\widetilde x_1,\dots,\widetilde x_N$ dual to the basis diagonalizing $g$. Then each~$\widetilde x_k$ is a linear combination of $x_1,\dots,x_N$. Moreover, one can renumber the new variables such that the weight of $\widetilde x_k$ is the same as the weight of $x_k$, namely $q_k$.

The element $j_f$ is represented in the old and the new basis by the same diagonal matrix. The given element $g$ acts of $\widetilde x_k$ just by a rescaling. Therefore it is enough to show the proposition for~$g$ belonging to the maximal group of diagonal symmetries.

To prove the propositions for $g \in G_f^d$ it is enough to prove the inequality for any $g \in G_f^{\rm gr}$ with~$N_g=0$ and $f$, such that the graph $\Gamma_f$ has only one connected component.

We have
\begin{equation}\label{eq: sum of weights}
	 \sum_{k=1}^N q_k = (1,\dots,1) E_f^{-1} \boldsymbol{1} = (1,\dots,1) \bigl(E_f^{\mathsf{T}}\bigr)^{-1} \boldsymbol{1} = \sum_{k=1}^N q_k^{\mathsf{T}}.
\end{equation}

For a given $g$ assume $s$, such that $\overline g = E_f^{-1}s$ as in Proposition~\ref{prop: age by matrix}. None of $s_k =0$ because~$N_g=0$. We have
\begin{align*}
	\age(g) &= ( (1,\dots,1) E_f^{-1} s )^{\mathsf{T}} = s^{\mathsf{T}} \bigl(E_f^{\mathsf{T}}\bigr)^{-1} \boldsymbol{1} = \sum_{k=1}^N s_k q_k^{\mathsf{T}}.
\end{align*}

First assume $f$ is not star-shaped. Then
\[
 \sum_{k=1}^N s_k q_k^{\mathsf{T}} \ge \sum_{k=1}^N q_k^{\mathsf{T}}
\]
because every $s_k \ge 1$ and $q_i^{\mathsf{T}}$ are all positive. Combining with equation~\eqref{eq: sum of weights} we get the inequality claimed. Moreover it is obvious that the equality is only reached if $s_k=1$ for all $k$. This is equivalent to the fact that $g = j_f$.

Now let $f$ be star-shaped. We have $q_1^{\mathsf{T}} = 0$ and $q_k^{\mathsf{T}} > 0$ for $k=2,\dots,N$. By the same reasoning as above it is enough to consider $\overline g = E_f^{-1}s$ with $s_2 = \dots = s_N = 1$. Then
\[
 \overline g = \left(\frac{s_1}{a_1}, \frac{1}{b_2}\frac{a_1-s_1}{a_1}, \dots, \frac{1}{b_{N}}\frac{a_1-s_1}{a_1} \right)
\]
for some $s_1 = 1,\dots,a_1-1$.
If $f$ defines an isolated singularity, it should have at least one summand $x_2^{r_2}\cdots x_N^{r_N}$ for some nonnegative $r_2,\dots,r_N$. The quasihomogeneity and the $g$-invariance conditions on this summand give
\[
 \sum_{k=2}^N \frac{r_k}{b_k}\left(1 - \frac{1}{a_1}\right) = 1, \qquad \sum_{k=2}^N \frac{r_k}{b_k}\frac{a_1-s_1}{a_1} \in \ZZ_{\ge 1}.
\]
These two conditions can only hold when $s_1=1$.
\end{proof}

\begin{Remark}
The proposition above holds for any invertible polynomial without the Calabi--Yau condition too. However for noninvertible polynomial without Calabi--Yau condition the proposition does not hold in general. If particular for ${f = x_1^{10} + x_1\bigl(x_2^2 + x_3^2 + x_4^2\bigr)}$ and $\overline g = \bigl(\frac{1}{5},\frac{2}{5},\frac{2}{5},\frac{2}{5}\bigr)^{\mathsf{T}}$ we have $\overline q = \bigl(\frac{1}{10},\frac{9}{20},\frac{9}{20},\frac{9}{20}\bigr)$ and $\age(g) - \sum_{k=1}^4 q_k = -\frac{1}{20}$.
\end{Remark}

\section{The total space}\label{section: total space}
Consider the quotient ring
\begin{equation*}
	\Jac(f) := \QR{\CC[x_1,\dots, x_N]}{\bigl(\frac{\partial f}{\partial x_1},\dots,\frac{\partial f}{\partial x_N}\bigr)}.
\end{equation*}
It is a finite-dimensional $\CC$-vector space whenever $f$ defines an isolated singularity.
Call it {\em Jacobian algebra} of $f$ and set $\mu_f:=\dim_\CC\Jac(f)$~-- the Milnor number of $f$.

We will assume an additional convention:
for the constant function $f = 0$ set $\Jac(f) := \CC$, $\mu_f := 1$.

\subsection{Grading}

The reduced weights $q_1,\dots,q_N$ of $f$ define the $\QQ$-grading on $\CC[x_1,\dots,x_N]$.
Introduce the $\QQ$-grading on $\Jac(f)$ by setting
\[
	\deg\bigl(\bigl[x_1^{\alpha_1}\cdots x_N^{\alpha_N}\bigr]\bigr) := \alpha_1q_1 + \dots + \alpha_Nq_N.
\]
Let $\phi_1,\dots,\phi_\mu$ be the classes of monomials, generating $\Jac(f)$ as a $\CC$-vector space.
We say that $X \in \Jac(f)$ is of degree $\kappa$ if it is expressed as a $\CC$-linear combination of degree $\kappa$ elements $\phi_\bullet$.

Denote by $\Jac(f)_\kappa$ the linear subspace of $\Jac(f)$ spanned by the degree $\kappa$ elements. Let the {\em Hessian} of~$f$ be defined as the following determinant:
\begin{equation*}
\hess(f):=\det \left(\frac{\partial^2 f}{\partial x_{i} \partial x_{j}}\right)_{i,j=1,\dots,N}.
\end{equation*}
Its class is nonzero in $\Jac(f)$.

\begin{Proposition}\label{prop: top degree}
	The maximal degree of a $\Jac(f)$-element is ${\widehat c}={\widehat c}(f) := \sum_{k=1}^N (1 - 2 q_k)$. Moreover we have
	\[
		\Jac(f)_{\widehat c} = \CC \langle [\hess(f)] \rangle.
	\]
\end{Proposition}
\begin{proof}
 See \cite[Section II]{AGV85}.
\end{proof}

\subsection{Pairing}
The algebra $\Jac(f)$ can be endowed with the $\CC$-bilinear nondegenerate pairing $\eta_f$ called \textit{residue pairing} (see \cite[Chapter~5]{GH94}, \cite[Section~5.11]{AGV85}).
The value $\eta_f([u],[v])$ is taken as the projection of $[u][v]$ to the top graded component $\Jac(f)_{\widehat c}$ divided by its generator $[\hess(f)]$. In particular, we have $\eta_f([1],[\hess(f)]) = 1$.

\begin{Proposition}\label{prop: grading of Jac(f)}

For any $\beta$, such that $0 \le \beta \le {\widehat c}$ the perfect pairing $\eta_f$ induces an equivalence
\begin{gather*}
\phi_{f,\beta}\colon\ \Jac(f)_{\beta} \cong (\Jac(f)_{{\widehat c} - \beta})^\vee,\qquad
[p]\mapsto \eta_f([p],-),
\end{gather*}
where $(-)^\vee$ stands for the dual vector space.
\end{Proposition}
\begin{proof}
 See \cite[Section II]{AGV85}.
\end{proof}

\subsection{The total space}
For each $g\in \GL_f$, fix a generator of a one-dimensional vector space $\Lambda(g):=\bigwedge^{d_g}\bigl(\CC^N/\Fix(g)\bigr)$. Denote it by $\xi_g$.

For $g\in G_f^d$, it is standard to choose the generator to be the wedge product of $x_k$ with $k\in I^c_g$ taken in increasing order.

Define $\B_{\rm tot}(f)$ as the $\CC$-vector spaces of dimension $\sum_{g \in \GL_f} \dim \Jac(f^g)$
\begin{equation}\label{eq: Btot in jac sectors}
	\B_{\rm tot}(f) := \bigoplus_{g \in \GL_f} \Jac(f^g) \xi_g.
\end{equation}
Each direct summand $\Jac(f^g) \xi_g$ will be called the \textit{$g$-th sector}.
We will write just $\B_{\rm tot}$ when the polynomial is clear from the context.

\begin{Remark}
Note that for $g,h \in G$, such that $\Fix(g) = \Fix(h)$, we have $f^g = f^h$. Then $\Jac(f^g) = \Jac(f^h)$, but the formal letters $\xi_g \neq \xi_h$ help to distinguish $\Jac(f^g) \xi_g$ and $\Jac(f^h)\xi_h$, such that $\Jac(f^g) \xi_g \oplus \Jac(f^h)\xi_h$ is indeed a direct sum of dimension $\dim\Jac(f^g)+\dim\Jac(f^h)$.
\end{Remark}

\subsection{B-model group action}

Note that an element $h\in\GL_f$ induces a map
\[h\colon \ \Fix(g)\to \Fix\bigl(hgh^{-1}\bigr)\qquad \text{and hence}\qquad h\colon\ \Lambda(g)\to \Lambda\bigl(hgh^{-1}\bigr).\] Since we have fixed the generators $\xi_\bullet$, the latter map provides a constant $\rho_{h,g}\in\CC^*$ such that $h ( \xi_{g} ) = \rho_{h,g} \xi_{hgh^{-1}}$. We have
\begin{equation}\label{eq: rho relation}
 \rho_{h_2,h_1gh_1^{-1}}\rho_{h_1,g}=\rho_{h_2h_1,g}.
\end{equation}

Note, that if $g,h\in G_f^d$ or, more generally, if $g$ and $h$ commute $\rho_{h,g}$ is independent of the choice of the generators since $g=hgh^{-1}$. More precisely, in this case it could be computed as follows. Let $\lambda_k$, $\lambda'_k$ be the eigenvalues of $h$ and $g$ in their common eigenbasis, then
\begin{equation*}
 \rho_{h,g}=\prod_{\substack{ k=1,\dots,N \\ \lambda'_k \neq 1}} \lambda_k.
\end{equation*}

We define the action of $\GL_f$ on $\B_{\rm tot}$ by
\begin{align*}
h^*([p(\bx)]\xi_{g})=\rho_{h,g}\bigl[p\bigl(h^{-1}\cdot\bx\bigl)\bigr]\xi_{hgh^{-1}}.
\end{align*}
This is indeed a group action, i.e., $(h_2h_1)^* = h_2^* \cdot h_1^*$. Indeed, using equation~\eqref{eq: rho relation} we get
\begin{align*}
(h_2h_1)^*[p(\bx)]\xi_{g} &=\rho_{h_2h_1,g}\bigl[p\bigl((h_2h_1)^{-1}\cdot\bx\bigr)\bigr]\xi_{(h_2h_1)g(h_2h_1)^{-1}}
\\ &=\rho_{h_2,h_1gh_1^{-1}}\rho_{h_1,g}\bigl[p\bigl(h_1^{-1}h_2^{-1}\cdot\bx\bigr)\bigr]\xi_{h_2h_1gh_1^{-1}h_2^{-1}}
\\ &=h_2^*\bigl(\rho_{h_1,g}\bigl[p\bigl(h_1^{-1}\cdot\bx\bigl)\bigl]\xi_{h_1gh_1^{-1}}\bigl)=h_2^*h_1^*([p(\bx)]\xi_{g}).
\end{align*}

Note that, in particular, if $g,h\in G_f^d$ then $h$ acts on $\xi_g$ by
\begin{equation*}
	h\colon\ \xi_g\mapsto h^{*} (\xi_g) := \prod_{k\in I_g^c} h_k\,\cdot \xi_g.
\end{equation*}

\begin{Example}\label{ex: J--action and bigrading}
Because $I_\id^c = I_{j_f} = \varnothing$, we have
\[
 h^*(\xi_\id) = \xi_\id \qquad \text{and} \qquad h^*(\xi_{j_f}) = \det(h) \xi_{j_f}
\]
for any $h \in \GL_f$.
Similarly for any $[p]\xi_g$ with a homogeneous $p \in \CC[x_1,\dots,x_N]$ and $g\in \GL_f$ we have
\[
(j_f)^*([p]\xi_g) = \epi\biggl[-\deg(p) + \sum_{k \in I_g^c} q_k\biggr] \cdot [p]\xi_g.
\]
\end{Example}

For a finite $G \subseteq \GL_f$ put
\[\B_{{\rm tot},G}:=\bigoplus_{g \in G} \Jac(f^g) \xi_g\subset \B_{\rm tot}\]
and define the \textit{B-model state space} $\B(f,G)$ by
$\B(f,G) := ( \B_{{\rm tot},G} )^G$. Namely, the linear span of the $\B_{\rm tot}$ vectors that are invariant with respect to the action of all elements of~$G$.

\begin{Remark}
 In the literature (see, for example, \cite{Muk}) a different definition could be found where the sum is taken over the representatives of the conjugacy classes of $G$ and the invariants in each sector are taken with respect to the centralizer of the corresponding~$g$. The two definitions are in fact equivalent in the same way as in \cite[Proposition~42]{BI22}.
\end{Remark}

\begin{Example}\label{example: Fermat type state space}
	Let $f = x_1^{a_1}$~-- the Fermat type polynomial. Assume $a_1 = rm$ and consider $G$ to be generated by $g = (1/r)$. We have
	\begin{align*}
		\B_{\rm tot} &= \CC \bigl\langle [1]\xi_{\id}, [x_1]\xi_{\id}, \dots,\bigl[x_1^{rm-2}\bigr] \xi_\id \bigr\rangle \oplus \CC\langle [1]\xi_{g},\dots,[1]\xi_{g^{r-1}} \rangle.
	\end{align*}
	Because $I_g^c = I_{g^k}^c = \{1\}$, we have
	\[
		 \bigl(g^k \bigr)^*(\xi_{g^l}) = \exp\left( 2 \pi {\rm i} \cdot \frac{k}{r}\right) \xi_{g^l}.
	\]
	However $\bigl(g^k\bigr)^*\bigl(\bigl[x_1^l\bigr]\bigr) = \exp\bigl(-2 \pi {\rm i} \cdot \frac{kl}{r}\bigr) \bigl[x_1^l\bigr]$ and the $G$-invariant monomials are $x_1^{rn}$ with $n \in \ZZ$. This gives
	\begin{align*}
		\B(f,G) &= \CC \bigl\langle [1]\xi_{\id},[x_1^r]\xi_{\id}, \dots, \bigl[x_1^{r(m-1)}\bigr]\xi_{\id}\bigr\rangle.
	\end{align*}
\end{Example}

\subsection{Bigrading}\label{section: bigrading}
The following operators $q_l, q_r \colon \B_{\rm tot} \to \QQ$ were first introduced in \cite{IV90} giving the bigrading we use.

For any homogeneous $p \in \CC[x_1,\dots,x_N]$ define for $[p]\xi_g$ its {\itshape left charge} $q_l$ and {\itshape right charge} $q_r$ to be
\begin{equation}\label{eq: left right charges}
 (q_l, q_r) = \biggl(\deg p - \sum_{k \in I_g^c} q_{k} + \age{(g)}, \deg p - \sum_{k \in I_g^c} q_{k} + \age{\bigl(g^{-1}\bigr)} \biggr).
\end{equation}
This definition endows $\B_{\rm tot}$ with the structure of a $\QQ$-bigraded vector space.
For $u,v\in G_f^d$ it follows immediately that $q_\bullet (\xi_u) + q_\bullet (\xi_v) = q_\bullet (\xi_{uv})$ for $u,v\in G$, such that $I_u^c \cap I_v^c = \varnothing$.

This bigrading restricts to $\B(f,G)$ because $q_l$, $q_r$ commute with the action of $h^\ast$ for any $h \in \GL_f$, $h$ preserves the weights and $\age(g)=\age\bigl(hgh^{-1}\bigr)$.

\section{Hodge diamond of LG orbifolds}
Assume $N \ge 3$ and the reduced weight set of $f$ to satisfy the CY condition ${\sum_{k=1}^N q_k = 1}$ (see also Section~\ref{section: symmetries under CY}).

\begin{Proposition}
	For $f$ satisfying CY condition and $G$, such that $J \subseteq G \subseteq \SL_f$ both left and right charges $q_l$ and $q_r$ of any $Y \in \B(f,G)$ are integer.
\end{Proposition}
\begin{proof}
	Note that $q_r([p]\xi_g) = q_l([p]\xi_g) + (N-N_g) - 2 \age(g)$. Due to ${\rm age}(g)\in\ZZ$ the right charge~$q_r([p]\xi_g)$ is integral if and only if $q_l([p]\xi_g)$ is integral.
 It remains to recall that
 \[
 (j_f)^*([p]\xi_g)=\epi\biggl[-\deg(p) + \sum_{k \in I_g^c} q_k\biggr]\cdot [p]\xi_g
 \]
 by Example~\ref{ex: J--action and bigrading}. Hence for a class in $\B(f,G)$ we have $\epi[-q_l+\age(g)]=1$ and so $q_l$ is integer.
\end{proof}

The following two propositions state that the graded pieces of $\B(f,G)$ are organized into a~diamond when CY condition holds.

\begin{Proposition}\label{proposition: diamond under CY}
	Let $f$ be a quasihomogeneous polynomial satisfying the CY condition, let $G \subseteq \SL_f$ be a finite subgroup, and let $V^{a,b}$ stand for the bidegree $(a,b)$-subspace of $\B(f,G)$. We have
	\begin{enumerate}\itemsep=0pt
	 \item[$(i)$] $V^{a,b} = 0$ for $a < 0$ or $b < 0$;
	 \item[$(ii)$] $V^{0,0} \cong \CC$, generated by $[1]\xi_\id$;
	 \item[$(iii)$] $V^{a,b} = 0$ for $a > N-2$ or $b > N-2$;
	 \item[$(iv)$] $V^{N-2,N-2} \cong \CC$, generated by $[\hess(f)]\xi_\id$.
	\end{enumerate}

\end{Proposition}
\begin{proof}
	Assume $X = [p]\xi_g$ for $p$ being a polynomial fixed by~$g$.

	(i) If $g = \id$ we have $q_l(X) = q_r(X) = \deg p \ge 0$. For $g \neq \id$ we have $\age(g) \in \NN_{\ge 1}$. Rewriting
	\[
		q_l(X) = \deg p -\sum_{k \in I_g^c} q_{k} + \age(g),
	\]
	we see that $q_l(X) \ge 0$ because \smash{$\sum_{k \in I_g^c} q_{k}\le \sum_{k =1}^{N} q_{k}=1$}. Similarly for $q_r(X)$ by the same argument applied to $\age\bigl(g^{-1}\bigr)$.

	(ii) If $g = \id$ we have that $q_l(X) = q_r(X) = 0$ if and only if $\deg p = 0$. By Propositions~\ref{prop: top degree} and \ref{prop: grading of Jac(f)}, we have $[p] = \alpha [1]$ in $\Jac(f)$ for some constant $\alpha \in \CC$.

	For $g \neq \id$, we just saw that $\age(g)\ge 1$ and \smash{$\sum_{k \in I_g^c} q_{k}\le 1$}, so $q_l(X) = q_r(X) = 0$ is achieved only if $\deg p=0$, $N_g=0$ and $\age(g)=\age(g^{-1})=1$, which implies \[N=\age(g)+\age\bigl(g^{-1}\bigr)+N_g=2.\]

	(iii) If $g=\id$, the statement follows from Proposition~\ref{prop: top degree} as given the CY condition we have $\hat{c}=N-2\sum_k q_k=N-2$.

 For $g \neq \id $, apply the same proposition again to estimate $\deg p$ in $\Jac(f^g)$. Namely, it gives
	\[
		q_l(X) \le N_g - 2\sum_{k \in I_g}q_k - \sum_{k \in I_g^c} q_{k}+ \age(g) = N_g - \sum_{k \in I_g}q_k -1 + \age(g).
	\]
	At the same time, we have $N_g + \age(g) = N - \age\bigl(g^{-1}\bigr) \le N-1$ because $\age\bigl(g^{-1}\bigr) \in \NN_{\ge 1}$. Combining this with the inequality above we get
\begin{equation}\label{eq: inequality on q_l}
		q_l(X) \le N - 2 - \sum_{k \in I_g}q_k \le N - 2.
\end{equation}
	One gets in the similar way that $q_r(X) \le N-2$.

	(iv) If $g = \id$, by Proposition~\ref{prop: top degree}, we see that $q_l([\hess(f)]\xi_\id)=q_r([\hess(f)]\xi_\id)=N-2$. If~$g \neq \id$, $q_l(X) = q_r(X) = N-2$, then equation~\eqref{eq: inequality on q_l} implies that \smash{$\sum_{k \in I_g}q_k=0$} and, hence, $N_g=0$ and $\age(g)=\age\bigl(g^{-1}\bigr)=1$, which altogether implies \[N=\age(g)+\age\bigl(g^{-1}\bigr)+N_g=2.\tag*{\qed} \]\renewcommand{\qed}{}
\end{proof}

We now construct two symmetries of $\B_{\rm tot}$. The \textit{horizontal} morphism $\Psi$ and the \textit{vertical} morphism $\Phi$.

Consider the direct sum decomposition of $\B_{\rm tot}$ as in equation~\eqref{eq: Btot in jac sectors}. We first extend the $\phi_{f,\beta}$ isomorphism of Proposition~\ref{prop: grading of Jac(f)} to $\B_{\rm tot}$ in the following way. The hessian matrix of $f^g$ viewed coordinate free is a bilinear form on the tangent bundle of $\Fix(g)$. Therefore, its determinant $\hess(f^g)$ is canonically an element of $\bigl(\Lambda^{N_g} \Fix(g)^\vee\bigr)^{\otimes 2}\otimes \mathbb{C}[\Fix(g)]$. Fix a generator \[\xi^\vee_g:=\iot_{\xi_g}{\rm d}x_1\wedge\cdots\wedge {\rm d}x_N\in\Lambda^{N_g} \Fix(g)^\vee,\] where $\iot$ is the interior product operator and let the generator of $\bigl(\Lambda^{N_g}\Fix(g)^\vee\bigr)^{\otimes 2}$ to be $(\xi^\vee_g)^{\otimes 2}$. This choice allows us to fix $\hess(f^g)$ as a function on $\Fix(g)$ and, hence fix a pairing $\eta_{f^g}$ for the $g$-sector. As in Proposition~\ref{prop: grading of Jac(f)}, this in turn defines an isomorphism $\phi_{f^g,\beta}$ on each sector.
Now the vertical morphism $\Phi$ is the direct sum of these isomorphisms acting on each sector of $\B_{\rm tot}$
\[
	\Phi := \bigoplus_{g \in \GL_f,\beta\in\QQ} \phi_{f^g,\beta} \colon \ \B_{\rm tot}\to \B_{\rm tot}^\vee.
\]
It is an isomorphism restricted to $\B_{{\rm tot},G}$ for any finite $G$ because each of $\phi_{f^g,\beta}$ is an isomorphism.

Define the horizontal morphism $\Psi$ to act on the $g$-th sector by $\Psi([p] \xi_g) := [p] \xi_{g^{-1}}$.
Extend it by linearity to all $\B_{\rm tot}$.
This is an isomorphism because \smash{$f^g = f^{g^{-1}}$} and \smash{$\Jac(f^g) = \Jac\bigl(f^{g^{-1}}\bigr)$}.
\begin{Proposition}\label{proposition: Serre and Hodge under CY}\quad
\begin{itemize}\itemsep=0pt
\item[$(1)$] The maps $\Phi$ and $\Psi$ are well defined on $\B(f,G)$ for any finite $G \subseteq \SL_f$.
\item[$(2)$]	For $f$ satisfying CY condition and a finite $G \subseteq \SL_f$, let $V^{a,b}$ stand for the bidegree $(a,b)$-subspace of $\B(f,G)$.
	Then the maps $\Psi$ and $\Phi$ induce the $\CC$-vector spaces isomorphisms
	\[
		V^{a,b} \cong V^{b,a} \qquad \text{and} \qquad V^{a,b} \cong \bigl(V^{N-2-b,N-2-a}\bigr)^\vee.
	\]
\end{itemize}
\end{Proposition}
\begin{proof}
1) The map $\Psi$ commutes with the $G$-action since $\Fix(g)=\Fix\bigl(g^{-1}\bigr)$. Hence $\Psi$ preserves the invariants.

To see that $\Phi$ commute with the $G$-action, recall first that, $f\bigl(h^{-1}\cdot \mathbf{x}\bigr)=f(\mathbf{x})$ and, hence, $f^g\bigl(h^{-1}\cdot \mathbf{x}\bigr)=f^{hgh^{-1}}(\mathbf{x})$.
Furthermore, since $\det(h)=1$ we have $\rho_{h,g}h^*\bigl(\xi_{hgh^{-1}}^\vee\bigr)= \xi^\vee_g$ and we can conclude that
\[
\rho_{h,g}^{2}\eta_{f^{hgh^{-1}}}\bigl(\bigl[p_1\bigl(h^{-1}\cdot\mathbf{x}\bigr)\bigr]\xi_{hgh^{-1}},
\bigl[p_2\bigl(h^{-1}\cdot\mathbf{x}\bigr)\bigr]\xi_{hgh^{-1}}\bigr)=
\eta_{f^{g}}([p_1(\mathbf{x})]\xi_{g},[p_2(\mathbf{x})]\xi_{g}).
\]
Now, by the definition of $G$-action we get
\[
\eta_{f^{hgh^{-1}}}(h^*([p_1(\mathbf{x})]\xi_{g}),h^*([p_2(\mathbf{x})]\xi_{g}))=
\eta_{f^{g}}([p_1(\mathbf{x})]\xi_{g},[p_2(\mathbf{x})]\xi_{g}).
\]
This implies the statement.

2) We have directly by the definition that $q_l(\Psi(X)) = q_r(X)$ and $q_r(\Psi(X)) = q_l(X)$. The first isomorphism follows.

To verify compatibility of $\Phi$ with the grading, note first, that $\widehat{c}(f^g)=\sum_{k\in I_g}(1-2q_k)$. Thus, by Proposition~\ref{prop: grading of Jac(f)} the left charge of $[\phi_{f^g,\deg p}(p)]\xi_g$ is given by
\begin{align*}q_l([\phi_{f^g,\deg p}(p)]\xi_g) &= \sum_{k\in I_g}(1-2q_k)-\deg p- \sum_{k \in I_g^c} q_{k} + \age{(g)}
\\ &=N_g-2\sum_{k\in I_g} q_k-\deg p- \sum_{k \in I_g^c} q_{k}+\bigl(N-N_g-\age{\bigl(g^{-1}\bigr)}\bigr)
\\ &=-2+\sum_{k\in I^c_g} q_k-\deg p +N-\age{\bigl(g^{-1}\bigr)}=N-2-q_r([p]\xi_g).
\end{align*}
The computation for the right charge is identical.
\end{proof}

Consider now two more special graded pieces of $\B(f,G)$.

\begin{Proposition}\label{proposition: hd0}
	For $f$ satisfying CY condition and a finite $G$, such that $J \subseteq G \subseteq \SL_f$, let~$V^{a,b}$ stand for the bidegree $(a,b)$-subspace of $\B(f,G)$.
	Then
	\begin{enumerate}\itemsep=0pt
	 \item[$(1)$] $V^{N-2,0} \cong \CC$, generated by $[1]\xi_{j_f^{-1}}$,
	 \item[$(2)$] $V^{0,N-2} \cong \CC$, generated by $[1]\xi_{j_f}$.
	\end{enumerate}
\end{Proposition}
\begin{proof}
 One notes immediately that $[1]\xi_{j_f}$ and \smash{$[1]\xi_{j_f^{-1}}$} are non-zero in $\B(f,G)$ and belong to~$V^{0,N-2}$ and $V^{N-2,0}$ respectively. By Proposition~\ref{proposition: Serre and Hodge under CY} it is enough to show one of the statements.

 \begin{Lemma}
 $\dim V^{0,N-2} = \lbrace g \in G \mid \age(g) = 1$, $ N_g = 0 \rbrace$.
 \end{Lemma}
 \begin{proof}
 Let $[p]\xi_g \in V^{0,N-2}$. It follows from equations~\eqref{eq: age g and inverse} and \eqref{eq: left right charges} that $\age(g) = 1 - N_g/2$. The statement follows by Proposition~\ref{prop: integral age}.
 \end{proof}

 Under the CY condition for $g \in \GL_f \backslash \lbrace \id \rbrace$ of finite order and with integral $\age(g)$ we have by Proposition~\ref{prop: age greater than weights} that $\age(g) \ge 1$ with the equality being reached only for $g = j_f$. This completes the proof.
\end{proof}
This completes the proof of Theorem~\ref{theorem: main}.

For a fixed pair $(f,G)$ set{\samepage
\begin{align*}
	h^{a,b} := \dim_\CC \left\lbrace X \in \B(f,G) \mid (q_l(X),q_r(X)) = (a,b)\right\rbrace
\end{align*}
and denote $D := N-2$.}

It follows from the propositions above that for $f$ satisfying CY condition and $G$ such that $J \subseteq G \subseteq \SL_f$, the numbers $h^{a,b}$ form a diamond,
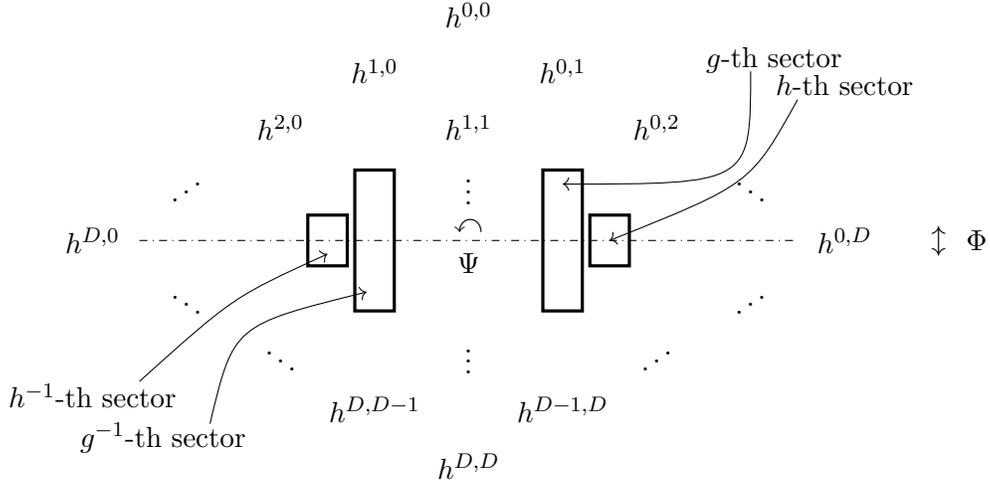
\[
\begin{tikzpicture}
\begin{scope}[x={(-1.25cm,-.75cm)},y={(1.25cm,-.75cm)}]
\node at (0,0) {$h^{0,0}$};
\node at (0,1) {$h^{0,1}$};
\node at (1,0) {$h^{1,0}$};
\node at (1,1) {$h^{1,1}$};
\node at (2,0) {$h^{2,0}$};
\node at (0,2) {$h^{0,2}$};
\node at (3,0) {$\rdots$};
\node at (0,3) {$\ddots$};
\node at (1.5,1.5) {$\vdots$};
\node at (4,0) {$h^{D,0}$};
\node at (0,4) {$h^{0,D}$};
\node at (4,1) {$\ddots$};
\node at (1,4) {$\rdots$};
\node at (3,3) {$\vdots$};
\node at (4,2) {$\ddots$};
\node at (2,4) {$\rdots$};
\node at (4,3) {$h^{D,D-1}$};
\node at (3,4) {$h^{D-1,D}$};
\node at (4,4) {$h^{D,D}$};
\node [label=below:$\Psi$] at (2-.125, 2-.125) {$\curvearrowleft$};
\node [label=right:$\Phi$] at (-.5, 4 + .5) {$\updownarrow$};
\draw[very thick] (1-.105-.125,2+.105-.125) rectangle (2+.105+.125,3-.105+.125);
\draw[very thick] (2-.105-.125,1+.105-.125) rectangle (3+.105+.125,2-.105+.125);
\draw[very thick] (1+0.25-.105-.225,3-0.25+.105-.225) rectangle (1+0.25+.105+.225,3-0.25-.105+.225);
\draw[very thick] (3-.5+0.25-.105-.225,1+.5-0.25+.105-.225) rectangle (3-.5+0.25+.105+.225,1+.5-0.25-.105+.225);
\draw[dash dot] (4-.25,.25) -- (.25,4-.25);
\draw[->] (-1,2) .. controls (0,3) .. (1,2);
\node at (-1-.125-.105,2+.125-.105) {$g$-th sector};
\node at (-1-.5+.125,2+.5+.125) {$h$-th sector};
\draw[->] (-1,2+.5) .. controls (0,3) .. (1+0.25,3-0.25);
\node at (4+1+.125+.25,4-2-.125+.25) {$g^{-1}$-th sector};
\node at (4+1+.5-.125,4-2-.5-.125) {$h^{-1}$-th sector};
\draw[->] (4+1-.25+.25,4-2-.5) .. controls (4-.25,1+.25) .. (3-.5+0.25+.105+.225-.105-.105,1+.5-0.25-.105+.225);
\draw[->] (4+1,4-2+.25) .. controls (4,1+.5) .. (3,2-.105+.125-.105);
\end{scope}
\end{tikzpicture}
\]

Let us call the line $\bigl\lbrace h^{a,b} \mid a+b = D \bigr\rbrace$ the horizontal line and the line $\bigl\lbrace h^{a,b} \mid a-b = 0 \bigr\rbrace$ the vertical line.
The Hodge diamond $\bigl\lbrace h^{a,b} \bigr\rbrace_{a,b=0}^D$ has the following special properties
\begin{enumerate}\itemsep=0pt
 \item[(1)] The $g$-th sector of $\B(f,G)$ contributes as a line symmetric with respect to the horizontal line.
 \item[(2)] Every $g$-th sector of $\B(f,G)$ contributes together with a $g^{-1}$-th sector of $\B(f,G)$, located symmetrically with respect to the vertical line.
 \item[(3)] All the elements of the form $\xi_{j_f^k}$ contribute to the horizontal line. In particular, $h^{a,D-a} \ge 1$ for all $a=0,\dots,D$.
 \item[(4)] All the elements of the form $[p]\xi_\id$ contribute to the vertical line.
\end{enumerate}

\begin{Example}
Consider $f = x_1^2 x_2 + x_2^2 + x_2 x_3^6 + x_4^6 + x_1x_3^9$ and $G = \SL_f$. Then $G = J = \langle j_f \rangle$ with
\[
	\overline j_f = \left(\frac{1}{4}, \frac{1}{2}, \frac{1}{12}, \frac{1}{6} \right).
\]
The basis of $\B(f,G)$ is given by the elements
\begin{gather*}
		 \xi_{j_f^3},\qquad \xi_{j_f^5}, \qquad \xi_{j_f^7}, \qquad \xi_{j_f^9},
\qquad [x_1]\xi_{j_f^4},\qquad [x_1]\xi_{j_f^8},\qquad \bigl[x_4^2\bigr]\xi_{j_f^6},\qquad \bigl[x_3^4 x_4^4\bigr]\xi_\id,
		\\
		 \bigl[x_1 x_3 x_4^4\bigr]\xi_\id, \qquad \bigl[x_1 x_3^3 x_4^3\bigr]\xi_\id, \qquad \bigl[x_2 x_4^3\bigr]\xi_\id,\qquad \bigl[x_1^2 x_4^3\bigr]\xi_\id,\qquad \bigl[x_1 x_3^5 x_4^2\bigr]\xi_\id,\\
\bigl[x_2 x_3^2 x_4^2\bigr]\xi_\id,
	\qquad \bigl[x_1^2 x_3^2 x_4^2\bigr]\xi_\id, \qquad \bigl[x_2 x_3^4 x_4\bigr]\xi_\id,\qquad  [x_1 x_2 x_3 x_4 ]\xi_\id, \qquad \bigl[x_1^3 x_3 x_4\bigr]\xi_\id,\\
\bigl[x_1^2\bigr]\xi_\id, \qquad \bigl[x_2^2\bigr]\xi_\id.
\end{gather*}
all having the bigrading $(1,1)$, and the elements
\[
	\xi_{j_f}, \qquad \xi_{j_f^{11}}, \qquad [1] \xi_\id, \qquad \bigl[x_1 x_2^2 x_3 x_4^4\bigr]\xi_\id,
\]
having the bigrading $(0,2)$, $(2,0)$, $(0,0)$ and $(2,2)$, respectively.

One gets the following diamond:
\[
 \begin{tikzpicture}
 \begin{scope}[x={(-0.50cm,-.50cm)},y={(0.50cm,-.50cm)}]
 \draw (0,0) node {1};
 \draw (1,0) node {0};
 \draw (0,1) node {0};
 \draw (2,2) node {1};
 \draw (2,0) node {1};
 \draw (0,2) node {1};
 \draw (1,2) node {0};
 \draw (2,1) node {0};
 \draw (1,1) node {20};
 \end{scope}
\end{tikzpicture}
\]
\end{Example}

\begin{Example}
Let $f=x_1^5+x_2^5+x_3^5+x_4^5+x_5^5$ and $G=S\ltimes J$, where $S=\langle (1,2), (2,3)\rangle\subset S_5$ is the subgroup permuting first 3 variables and preserving the last two. Pick $\xi_{(1,2,3)}$ and $\xi_{(1,3,2)}$ in such a way that $(1,2)(\xi_{(1,2,3)})=\xi_{(1,3,2)}$; $\xi_{(1,2,3)j_f}$ and $\xi_{(1,3,2)j_f}$ in such a way that $(1,2)(\xi_{(1,2,3)j_f})=\xi_{(1,3,2)j_f}$ and so on. Then the basis of $\B(f,G)$ is
\[
\xi_{\id} \qquad \text{and} \qquad \bigl[x_1^3x_2^3x_3^3x_4^3x_5^3\bigr]\xi_{\id}
\]
in bidegrees $(0,0)$ and $(3,3)$, respectively;
\begin{gather*}
[x_1x_2x_3x_4x_5]\xi_{\id},\qquad\bigl[x_4^3x_5^2\bigr]\xi_{\id},\qquad\bigl[x_4^2x_5^3\bigr]\xi_{\id}, \qquad\bigl[x_1x_2x_3x_4^2\bigr]\xi_{\id},\qquad \bigl[x_1x_2x_3x_5^2\bigr]\xi_{\id},\\
\bigl[(x_1+x_2+x_3)^2\bigr](\xi_{(1,2,3)}+\xi_{(1,3,2)}),\qquad
\bigl[x_4^2\bigr](\xi_{(1,2,3)}+\xi_{(1,3,2)}),\\ \bigl[x_5^2\bigr](\xi_{(1,2,3)}+\xi_{(1,3,2)}),\qquad
[(x_1+x_2+x_3)x_4](\xi_{(1,2,3)}+\xi_{(1,3,2)}),\\ [(x_1+x_2+x_3)x_5](\xi_{(1,2,3)}+\xi_{(1,3,2)}),\qquad [x_4x_5](\xi_{(1,2,3)}+\xi_{(1,3,2)})
\end{gather*}
in bidegree $(1,1)$;
\[
\xi_{j_f} \qquad \text{and} \qquad \xi_{j_f^4}
\]
in bidegrees $(3,0)$ and $(0,3)$, respectively;
\[
\xi_{j_f^2}, \qquad \xi_{(1,2,3)j_f}+\xi_{(1,3,2)j_f}, \qquad \xi_{(1,2,3)j_f^2}+\xi_{(1,3,2)j_f^2}
\]
in bidegree $(2,1)$;
\[
\xi_{j_f^3}, \qquad \xi_{(1,2,3)j_f^3}+\xi_{(1,3,2)j_f^3}, \qquad \xi_{(1,2,3)j_f^4}+\xi_{(1,3,2)j_f^4}
\]
in bidegree $(1,2)$;
\begin{gather*}
\bigl[x_1^2x_2^2x_3^2x_4^2x_5^2\bigr]\xi_{\id},\qquad \bigl[x_1^2x_2^2x_3^2x_4^3x_5\bigr]\xi_{\id},\qquad\bigl[x_1^2x_2^2x_3^2x_4x_5^3\bigr]\xi_{\id}, \qquad\bigl[x_1^3x_2^3x_3^3x_4\bigr]\xi_{\id},\\
\bigl[x_1^3x_2^3x_3^3x_5\bigr]\xi_{\id},\qquad
\bigl[(x_1+x_2+x_3)^3x_4^3x_5\bigr](\xi_{(1,2,3)}+\xi_{(1,3,2)}),\\
\bigl[(x_1+x_2+x_3)^3x_4x_5^3\bigr](\xi_{(1,2,3)}+\xi_{(1,3,2)}),\qquad
\bigl[(x_1+x_2+x_3)x_4^3x_5^3\bigr](\xi_{(1,2,3)}+\xi_{(1,3,2)}),\\
\bigl[(x_1+x_2+x_3)^3x_4^2x_5^2\bigr](\xi_{(1,2,3)}+\xi_{(1,3,2)}),\qquad \bigl[(x_1+x_2+x_3)^2x_4^3x_5^2\bigr](\xi_{(1,2,3)}+\xi_{(1,3,2)}),\\ \bigl[(x_1+x_2+x_3)^2x_4^2x_5^3\bigr](\xi_{(1,2,3)}+\xi_{(1,3,2)})
\end{gather*}
in bidegree $(2,2)$.

This gives the following diamond:
\[
 \begin{tikzpicture}
 \begin{scope}[x={(-0.50cm,-.50cm)},y={(0.50cm,-.50cm)}]
 \draw (0,0) node {1};
 \draw (1,0) node {0};
 \draw (0,1) node {0};
 \draw (2,2) node {11};
 \draw (2,0) node {0};
 \draw (0,2) node {0};
 \draw (1,2) node {3};
 \draw (2,1) node {3};
 \draw (1,1) node {11};
 \draw (3,0) node {1};
 \draw (0,3) node {1};
 \draw (3,1) node {0};
 \draw (1,3) node {0};
 \draw (3,2) node {0};
 \draw (2,3) node {0};
 \draw (3,3) node {1};
 \end{scope}
\end{tikzpicture}
\]
\end{Example}

\subsection*{Acknowledgements}
The work of Alexey Basalaev was supported by International Laboratory of Cluster Geometry NRU HSE, RF Government grant, ag. no.~075-15-2021-608 dated 08.06.2021.
The authors are grateful to Anton Rarovsky for sharing the pictures from his bachelor thesis.
The authors are very grateful to the anonymous referees for many valuable comments.

\pdfbookmark[1]{References}{ref}
\LastPageEnding

\end{document}